%% file: arxiv_updated.tex
\documentclass[11pt]{amsart}

\usepackage[utf8]{inputenc}
\usepackage{graphicx}
\usepackage{float}
\usepackage{amsthm}
\usepackage{amsmath}
\usepackage{amssymb}
\usepackage[abs]{overpic}
\usepackage{enumerate}
\usepackage{verbatim}
\usepackage{caption}
\usepackage{url}
\usepackage{textcomp}
\usepackage{color}
\usepackage{hyperref}
\usepackage{multicol}
\usepackage{gensymb}
\usepackage{subcaption}
\usepackage{bm}
\usepackage{CJKutf8}
\usepackage{multirow}
\usepackage{tabularx}
\usepackage{transparent}
\usepackage{mathtools}
\usepackage{enumitem}
\usepackage{doi}
\usepackage{appendix}
\allowdisplaybreaks
\newcounter{descriptcount}
\newlist{enumdescript}{description}{1}
\setlist[enumdescript,1]{
  before={\setcounter{descriptcount}{0}
          \renewcommand*\thedescriptcount{\arabic{descriptcount}}},
        font={\bfseries\stepcounter{descriptcount}Case \thedescriptcount~}
}

\newtheorem{lemma}{Lemma}[section]

\theoremstyle{definition}
\newtheorem{definition}[lemma]{Definition}

\theoremstyle{plain}
\newtheorem{proposition}[lemma]{Proposition}

\numberwithin{equation}{section}
\usepackage[abbrev,nobysame]{amsrefs}
\title{Bifurcations of the H\'{e}non map with additive bounded noise}

\begin{document}

\author{Jeroen S.W. Lamb}
\address{Department of Mathematics, Imperial College London, 180 Queen's Gate, London SW7 2AZ, United Kingdom}
\address{International Research Center for Neurointelligence (IRCN), The University of Tokyo, 7-3-1 Hongo Bunkyo-ku, Tokyo, 
113-0033 Japan}
\address{Centre for Applied Mathematics and Bioinformatics, Department of Mathematics and Natural Sciences, Gulf University for Science and Technology, Halwally 32093, Kuwait}

\email{jsw.lamb@imperial.ac.uk}
\thanks{The first two authors were supported by the EPSRC grant EP/W009455/1 and EP/Y020669/1. The first author was also supported by the EPSRC grant EP/S023925/1. The third author was supported by the Project of Intelligent Mobility Society Design, Social Cooperation Program, UTokyo, JST Moonshot R \& D Grant Number JPMJMS2021, EPSRC grant EP/S515085/1}

\author{Martin Rasmussen}
\address{Department of Mathematics, Imperial College London, 180 Queen's Gate, London SW7 2AZ, United Kingdom}

\email{m.rasmussen@imperial.ac.uk}
\thanks{}

\author{Wei Hao Tey}
\address{Department of Mathematics, Imperial College London, 180 Queen's Gate, London SW7 2AZ, United Kingdom}
\address{International Research Center for Neurointelligence (IRCN), The University of Tokyo, 7-3-1 Hongo Bunkyo-ku, Tokyo, 
113-0033 Japan}
\email{w.tey18@imperial.ac.uk}


\begin{abstract}
    We numerically study bifurcations of attractors of the H\'{e}non map with additive bounded noise with spherical reach. The bifurcations are analysed using a finite-dimensional boundary map. We distinguish between two types of bifurcations: topological bifurcations and boundary bifurcations. Topological bifurcations describe discontinuous changes of attractors and boundary bifurcations occur when singularities of an attractor's boundary are created or destroyed. We identify correspondences between topological and boundary bifurcations of attractors and local and global bifurcations of the boundary map.
\end{abstract}

\maketitle

\section{Introduction}\label{sec:intro}
Bifurcation theory is pivotal for understanding qualitative changes in dynamical systems. While bifurcation theory is well developed in the context of deterministic dynamical systems, a corresponding theory for random dynamical systems remains relatively unexplored. Existing results on bifurcations in random dynamical systems mainly concern stochastic differential equations with unbounded noise~\cite{arnold1995random}, even though bounded noise is often realistic from a modelling perspective, for instance in settings where physically relevant fluctuations are intrinsically bounded~\cite{d2013bounded}. 

The aim of this paper is to demonstrate the applicability of a new finite-dimensional map to study bifurcations of attractors of random dynamical systems with bounded noise \cite{kourliouros2023persistence}. We focus on a randomised H\'{e}non map as a prototypical example. Recall that the H\'{e}non map $f:\mathbb{R}^2 \to \mathbb{R}^2$ is given by $f(x,y) = (1 - ax^2 + y,bx)$, depending on parameters $a,b\in\mathbb R$ \cite{henon1976two}. We consider this map with additive bounded noise, represented by the random difference equation
\begin{equation}\label{eq:randomdiff}
    z_{i+1}= (x_{i+1},y_{i+1}) = f(x_i,y_i) + \xi_i \quad \mbox{for } i \in \mathbb N_0\,,
\end{equation}
where the noise, denoted by $\xi:=(\xi_i)_{i\in \mathbb N_0}$, consists of i.i.d.~random variables that are supported on the $\varepsilon$-ball ($\xi_i \in \overline{B_{\varepsilon}(0)} = \{z \in \mathbb R^2 : \|z\|\leq \varepsilon\}$).

For such models with unbounded (e.g.~normally distributed) additive noise, initial conditions have non-zero probability of reaching any open subset of $\mathbb R^2$ (even in one time-step). However, in contrast, if noise is bounded, trajectories may be attracted to a minimal attractor, from which escape is not possible. We say that a subset $A \subset \mathbb R^2$ is a minimal attractor of the random dynamical system \eqref{eq:randomdiff} if it is attracting and minimal forward invariant. $A$ is attracting if its domain of attraction 
\begin{equation}\label{eq:domain_random}
\mathcal{D}(A) := \{z_0\in \mathbb{R}^2\mid \lim_{i\to \infty}\inf_{a\in A}\|z_i-a\| = 0,~\text{for all } (\xi_i)_{i\in \mathbb{N}_0}\}
\end{equation}
contains a neighbourhood of $A$.
$A$ is called forward invariant if $z_i \in A$ implies $z_{i+1} \in A$ for all noise realisations $\xi_i\in\overline{B_\varepsilon(0)}$, and it is minimal if no proper subset of $A$ is also forward invariant. Note that this notion of attractor is different from that of so-called random attractors~\cite{arnold1995random}, which are time-varying objects depending on noise realisations. The attractors we discuss in this paper are deterministic and coincide with the union of fibers of a corresponding random attractor. Crucially, our minimal attractors are interesting from a statistical perspective, since they support ergodic stationary measures~\cite{zmarrou2007bifurcations}. 

Understanding bifurcations of minimal attractors is important beyond the study of random systems. Namely, equations of motion of the form \eqref{eq:randomdiff} also arise in control systems where $(\xi_i)_{i\in \mathbb{N}_0}$ represents a control input. In this context, a minimal attractor is both an invariant control set \cite{colonius2012dynamics} and a reachable set~\cite{aubin1984differential,Deimling+1992}. Moreover, bifurcations of minimal attractors are also of relevance to the assessment of \emph{resilience} of the deterministic system $f$. In particular, \cite{mcgehee1988some,meyer2022intensity} propose the maximal additive perturbation amplitude 
for which the dynamics remains in the domain of attraction of an attractor $A$, as a measure of resilience. This amplitude coincides with a bifurcation of minimal attractors. 

In this paper, we confine the discussion to discrete-time dynamics. The corresponding set-valued dynamics in the continuous time is represented by differential inclusions  \cite{aubin1984differential}; see  \cite{zmarrou2007bifurcations,homburg2010bifurcations,homburg2013bifurcations} for some results on bifurcations of minimal attractors in that setting. 

Qualitative changes of minimal attractors can be observed when varying parameters in the random system \eqref{eq:randomdiff}, specifically the parameters $a,b$ of the H\'{e}non map $f$ or the noise amplitude $\varepsilon$. We consider two types of bifurcations: \emph{topological bifurcations}, where the minimal attractor undergoes a discontinuous change, for example, transitioning from a connected set to two disjoint sets; and \emph{boundary bifurcations}, where singularities on the boundary of the minimal attractor are created or destroyed (see~Figure~\ref{fig:top_bif_intro} and Figure~\ref{fig:boun_bif_intro}). 

\begin{figure}[ht]
    \begin{subfigure}[t]{0.49\textwidth}
        \centering
        \includegraphics[width=1\textwidth,trim={0 0 0 20},clip]{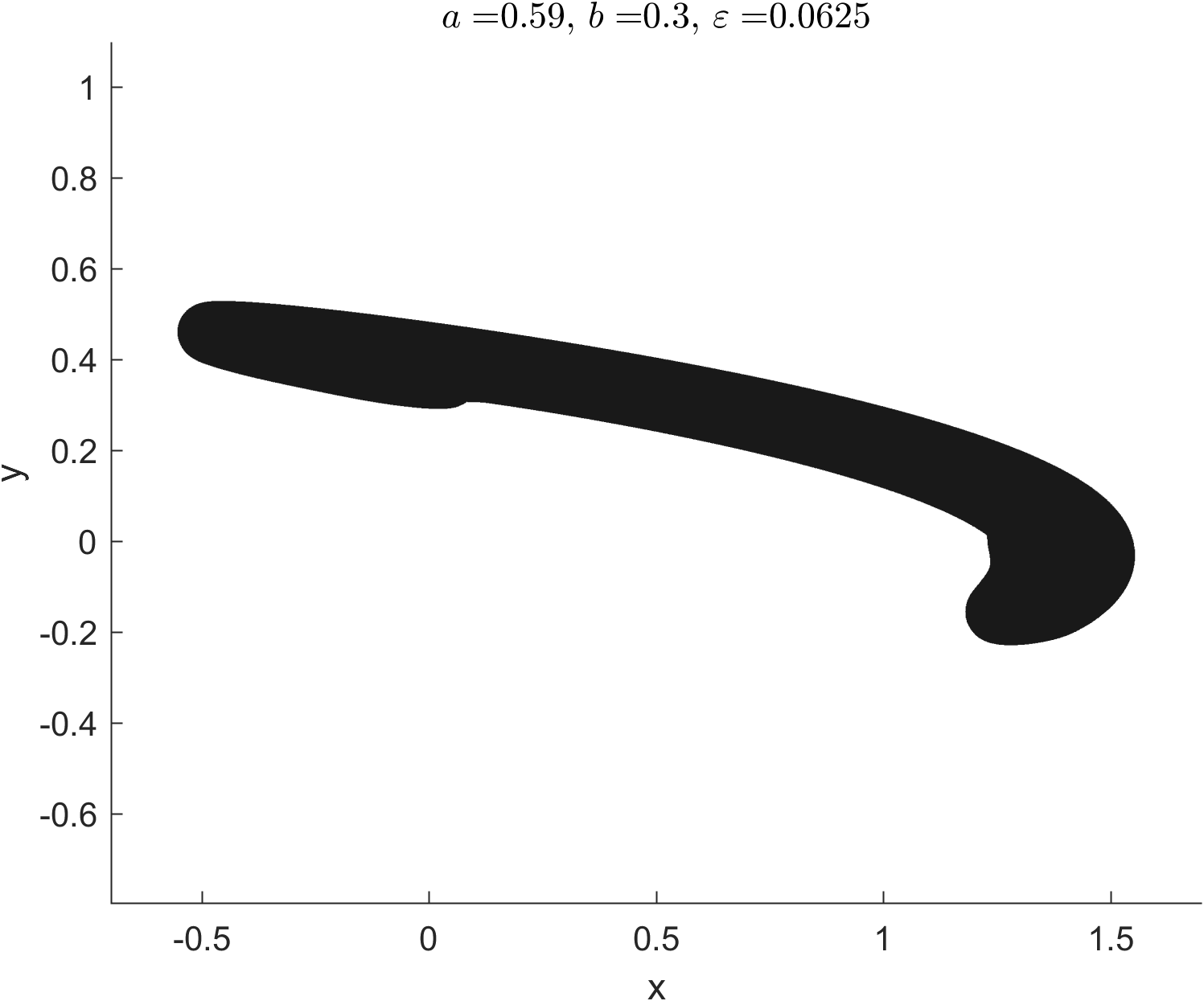}
        \captionsetup{width=\linewidth}\caption{$a = 0.59, b= 0.3, \varepsilon = 0.0625$.}
        \label{fig:a059_intro}
    \end{subfigure}
    \hfill
    \begin{subfigure}[t]{0.49\textwidth}
        \centering
        \includegraphics[width=1\textwidth,trim={0 0 0 20},clip]{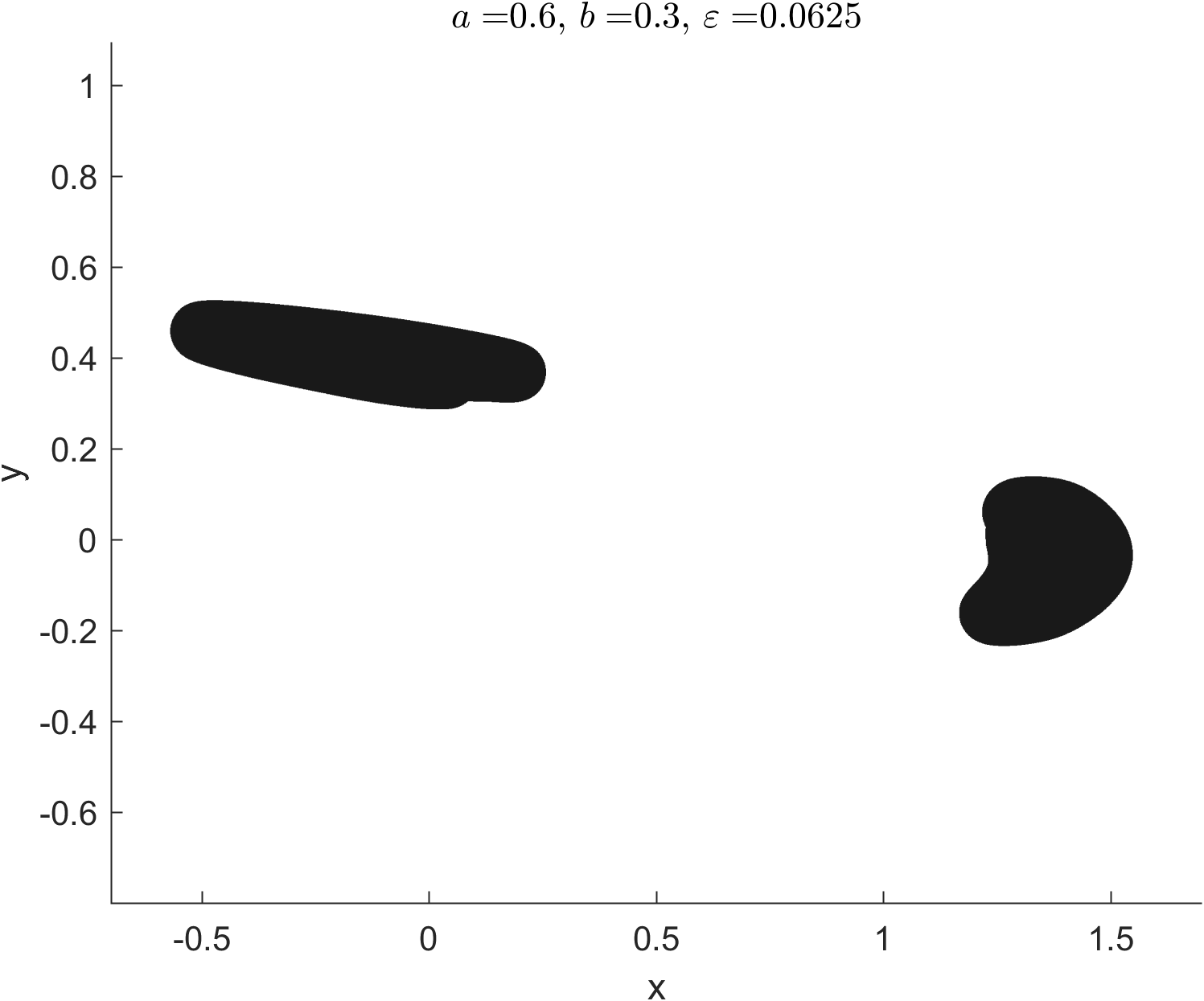}
        \captionsetup{width=\linewidth}\caption{$a = 0.6, b= 0.3, \varepsilon = 0.0625$.}
        \label{fig:a06_intro}
    \end{subfigure}
    \captionsetup{width=\linewidth}\caption{Numerical approximation of minimal attractors of the random H\'{e}non map \eqref{eq:randomdiff} with $b = 0.3$ and $\varepsilon = 0.0625$ and varying $a$ between $0.59$ and $0.6$, one observes a topological bifurcation from (a) one connected minimal attractor to (b) a minimal attractor that consists of two disjoint parts between which orbits alternate.}
    \label{fig:top_bif_intro}
\end{figure}

\begin{figure}[ht]
    \begin{subfigure}[t]{0.32\textwidth}
        \centering
        \includegraphics[width=1\textwidth,trim={0 0 0 0},clip]{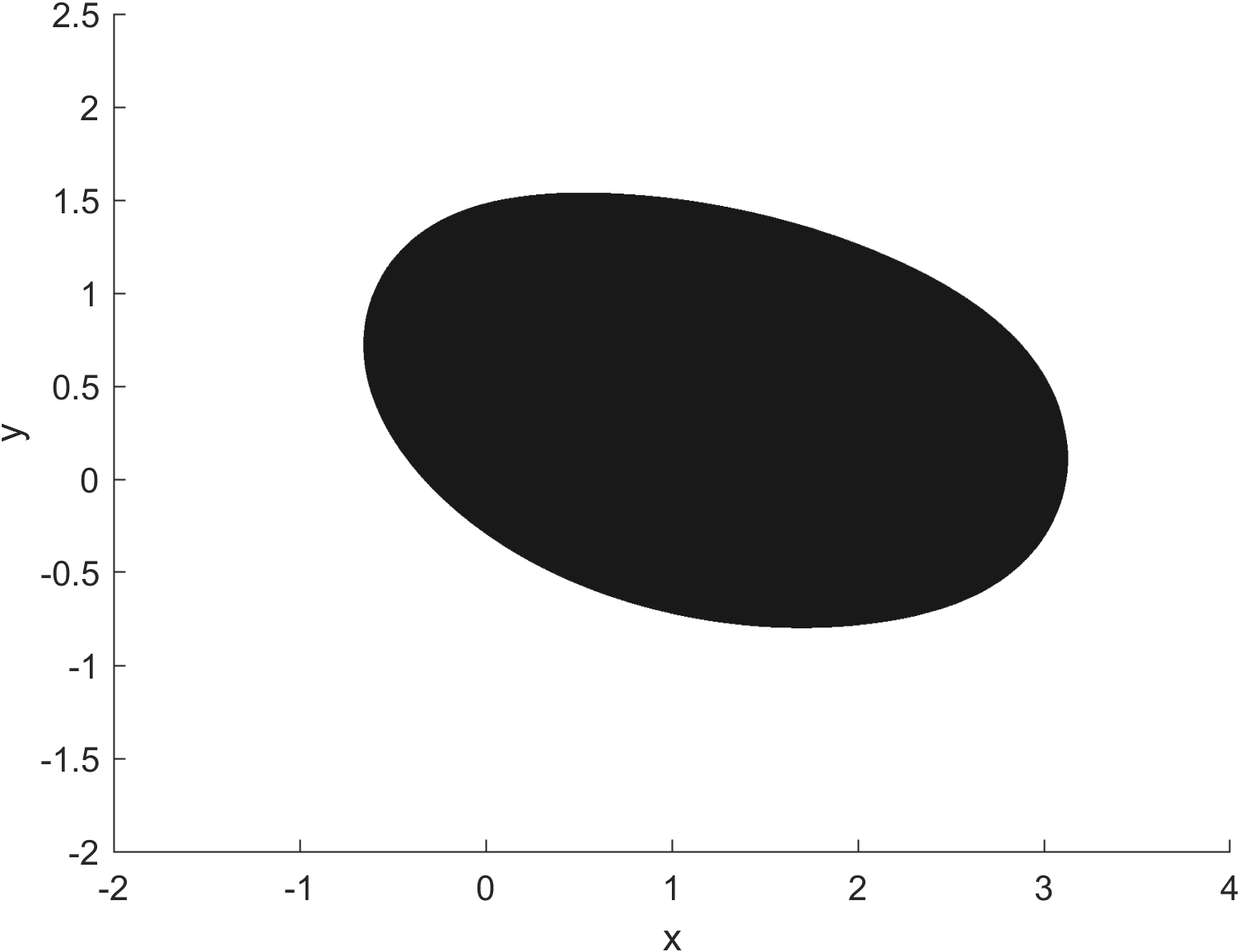}
        \captionsetup{width=\linewidth}\caption{$a = 0.06, b= 0.3, \varepsilon = 0.6$.}
        \label{fig:a006_intro}
    \end{subfigure}
    \hfill
    \begin{subfigure}[t]{0.32\textwidth}
        \centering
        \includegraphics[width=1\textwidth,trim={0 0 0 0},clip]{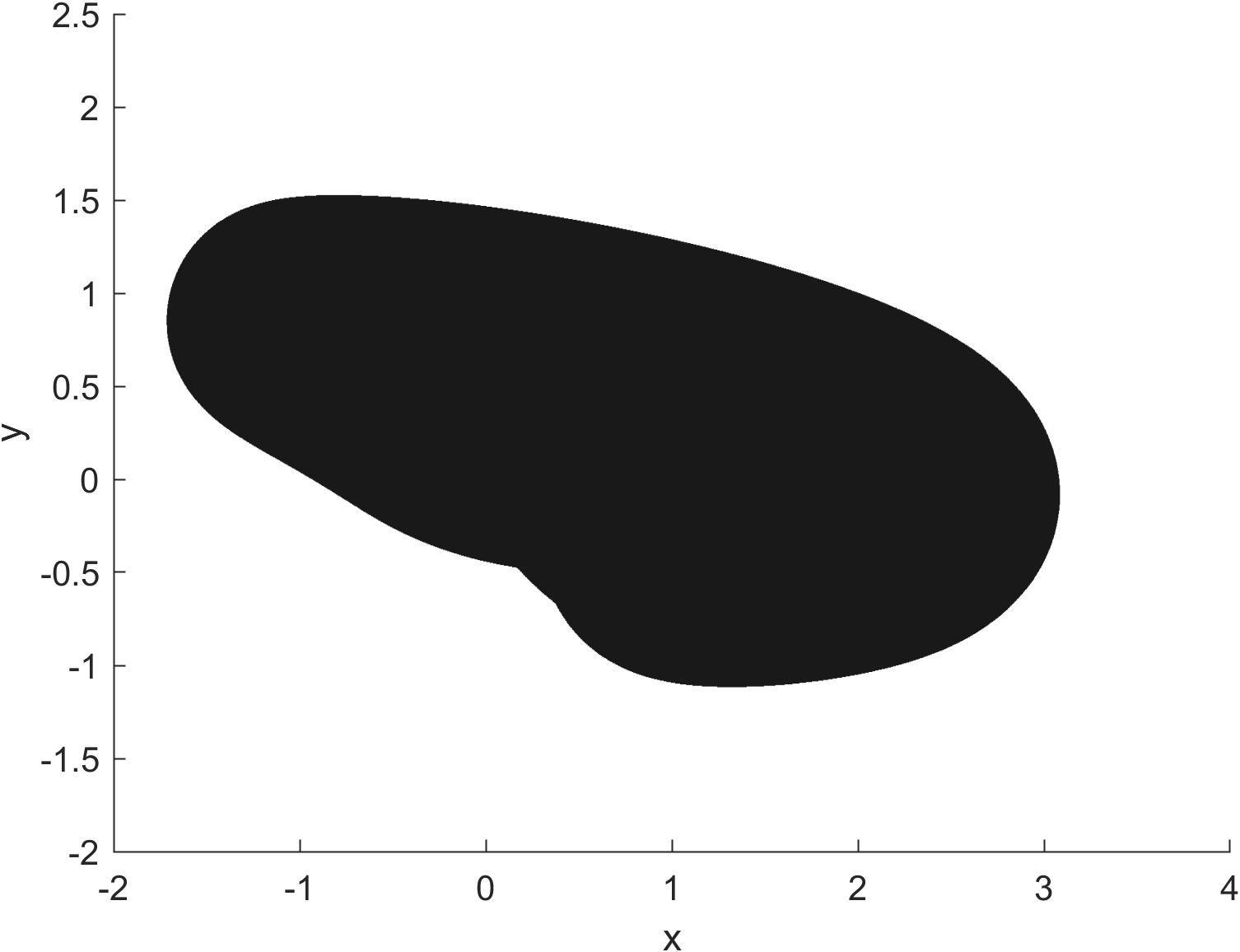}
        \captionsetup{width=\linewidth}\caption{$a = 0.18, b= 0.3, \varepsilon = 0.6$.}
        \label{fig:a018_intro}
    \end{subfigure}
    \hfill
    \begin{subfigure}[t]{0.32\textwidth}
        \centering
        \includegraphics[width=1\textwidth,trim={0 0 0 10},clip]{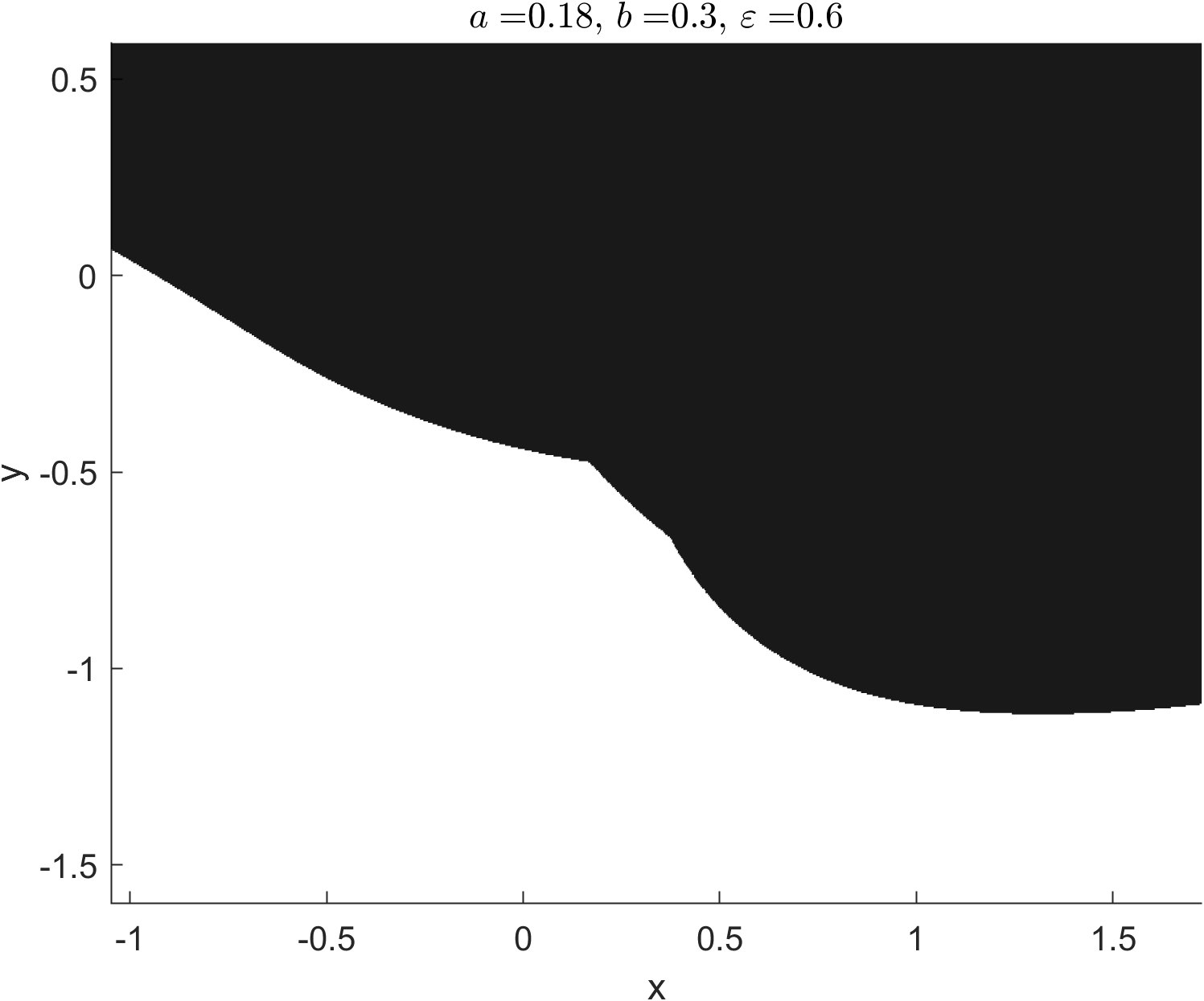}
        \captionsetup{width=\linewidth}\caption{Magnification of part of (b).}
        \label{fig:a018_intro_zoom}
    \end{subfigure}
    \captionsetup{width=\linewidth}\caption{Numerical approximation of the minimal attractor of the random H\'{e}non map \eqref{eq:randomdiff} with $b = 0.06$ and $\varepsilon = 0.6$. A boundary bifurcation is observed between $a = 0.06$ and $a = 0.18$, where (a) the boundary of the minimal attractor is smooth at $a=0.06$, and (b) singularities have appeared at $a=0.18$, see also the magnification in (c).}
    \label{fig:boun_bif_intro}
\end{figure}

\begin{figure}[h]
    \centering
    \includegraphics[width=.9\textwidth]{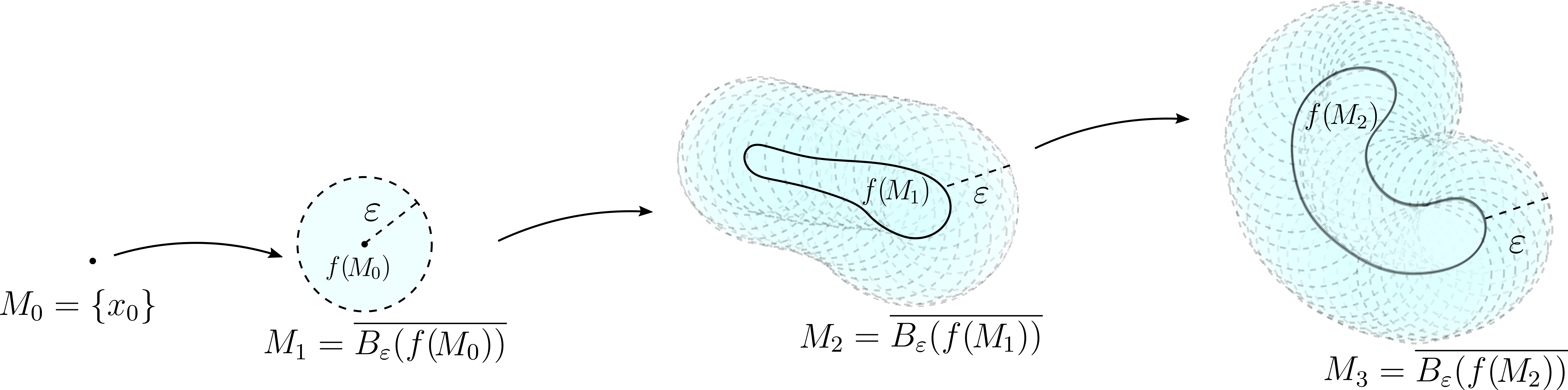}
    \captionsetup{width=\linewidth}\caption{Schematic illustration of iterations of the set-valued map $F(x) = \overline{B_{\varepsilon}(f(x))}$.}
    \label{fig:set-valued_illustrate}
\end{figure}

The compound behaviour of trajectories of the random system \eqref{eq:randomdiff}, considering all possible noise realisations, can be analysed through the dynamics of the set-valued map defined by $F(x) := \overline{B_{\varepsilon}(f(x))}$, with the natural extension to sets $S\subset \mathbb{R}^2$ as $F(S):=\bigcup_{x\in S} F(x)$. 
A schematic illustration of this set-valued dynamics is provided in Figure~\ref{fig:set-valued_illustrate}. 
Indeed, minimal attractors of the random system are (minimal) attractors of the set-valued map $F$. This observation has been used to produce Figures~\ref{fig:top_bif_intro} and \ref{fig:boun_bif_intro}, observing the convergence of iterations of initial sets by the set-valued map using set-oriented numerical methods of the software package GAIO \cite{dellnitz2001algorithms}. 
Unfortunately, however, the approximation of attractors by set-oriented methods is computationally expensive. Moroever, the set-valued point of view does not appear to aid much in the analysis of bifurcations, since set-valued dynamics evolves in the space of compact subsets of $\mathbb R^2$ which is not a Banach space. This prevents the use of standard bifurcation theory tools in this setting~\cite{colonius2012dynamics}. 

Circumventing the set-valued point of view, an alternative approach to the analysis of boundaries of attractors has been introduced in \cite{kourliouros2023persistence} via a new finite-dimensional \emph{boundary} map. 
The main aim of this paper is to provide a proof of concept on how the boundary map can be used to study bifurcations of minimal attractors of random diffeomorphisms with bounded additive noise.

Our numerical explorations reveal correspondences between topological bifurcations of random attractors with fold and heteroclinic bifurcations of the boundary map. We also identify two types of boundary bifurcations: the appearance of an isolated wedge singularity, associated with a change of stability type of a periodic orbit of the boundary map, and the emergence of an infinite cascade of wedge singularities converging towards a shallow singularity, associated with a non-transversal intersection between the unstable manifold of a saddle fixed point and the strong stable foliations of a stable periodic point of the boundary map. 

The remainder of this paper is organised as follows.
In Section~\ref{sec:setval}, we review some relevant concepts and results from set-valued dynamics. Then, in Section~\ref{sec:bm} we introduce the boundary map and discuss some of its elementary properties. The numerical explorations of the boundary map for the H\'enon map with bounded additive noise are presented in Section~\ref{sec:numerical}. We end this paper with an outlook towards future research in Section~\ref{sec:outlook}.

\section{The set-valued map and its bifurcations}\label{sec:setval}
 
Recall from the introduction above that the compound behaviour of trajectories of the random system (\ref{eq:randomdiff}) is encapsulated by a set-valued map $F:\mathcal{K}(\mathbb R^2) \to \mathcal{K}(\mathbb R^2)$, with $\mathcal{K}(\mathbb R^2)$ denoting the set of all non-empty compact subsets of $\mathbb R^2$ endowed with the (natural) Hausdorff metric $ $, defined by 
\begin{equation}\label{eq:setvaluedmap}
    F(X) := \bigcup_{x\in X} F(x)
    ,~~F(x):=\overline{B_\varepsilon(f(x))} := 
    \left\{f(x)+y \mid \|y\|\leq \varepsilon\right\}.
\end{equation}
If the probability distribution of the noise 
$(\xi_i)_{i\in \mathbb{N}_0}$ 
has a non-vanishing Lebesgue density on the $\varepsilon$-ball $\overline{B_{\varepsilon}(0)}$,
the support of a stationary measure of \eqref{eq:randomdiff} coincides with a minimal invariant set of the set-valued system $F$~\cite{zmarrou2007bifurcations,araujo2000attractors}.

As usual, we call a set $A\in \mathcal{K}(\mathbb{R}^2)$ an \emph{attractor} of $F$ if $A$ is $F$-invariant, $F(A) = A$ and there exists $\sigma > 0$ such that $\lim_{i \to \infty} d_H(F^i(B_{\sigma}(A)), A) = 0.$ It is a \emph{minimal attractor} if there is no proper subset of $A$ that is also an attractor. The domain of attraction 
of an attractor $A$ for $F$ consists of all sets in $\mathcal{K}(\mathbb{R}^2)$ that are attracted to $A$, i.e. $\left\{B\in \mathcal{K}(\mathbb{R}^2) ~|~ \lim_{i \to \infty}d(F^i(B),A) = 0\right\}$ 
where $d$ denotes the semi-Hausdorff distance. It follows from the special nature of the set-valued map in \eqref{eq:setvaluedmap} that a set $B$ is contained in the domain of attraction of $A$ for $F$ if and only if $B\subset \mathcal{D}(A)$ with $\mathcal{D}$ as defined in \eqref{eq:domain_random}.

For each attractor $A$, there is a corresponding dual repeller $A^*$ which is the complement of its domain of attraction, i.e. $A^* = \mathbb R^2 \backslash \mathcal D(A)$. The dual repeller $A^*$ is an invariant set of a map $F^*$ (dual to $F$) which associates a set with the union of all points whose $F$-images intersect this set. For $F$ given by (\ref{eq:setvaluedmap}), and $Y\subset \mathbb{R}^2$, the dual map $F^*$ can be written as 
\begin{equation}\label{eq:dual}
    F^*(Y) := \bigcup_{y\in Y}F^*(y),~~ F^*(y) := \bigcup_{x\in \overline{B_{\varepsilon}(y)}}\{f^{-1}(x)\}.
\end{equation}

Minimal attractors of the set-valued map can undergo a discontinuous change with respect to the Hausdorff metric, a so-called \emph{topological bifurcation}~\cite{lamb2015topological}*{Definition~1.1}, as the set-valued map is perturbed. A necessary condition for topological bifurcation is the collision of the minimal invariant set and its dual repeller~\cite{lamb2015topological}*{Theorem~6.1}.

Bifurcation theory in set-valued maps is notoriously challenging~\cite{colonius2012dynamics,hans}. Many powerful tools for traditional bifurcation analysis, such as the implicit function theorem, are not readily available for set-valued maps. Given that the collision occurs on the boundary, it is important to approximate and analyse the boundary of the minimal attractor $\partial A$ and its dual repeller $\partial A^*$. Studies of the general structure of the attractor's boundary have proven to be challenging~\cite{kourliouros2023persistence,lamb2020boundaries,lamb2021boundaries}.

\section{The boundary map}\label{sec:bm}
We now consider a set $M$ with smooth boundary $\partial M$. Its \emph{unit normal bundle} consists
of the pairs $(x,n(x))$ of boundary points $x\in\partial M$ and the unit normal $n(x)$ to its tangent space $T_x\partial M$. As there are two conventions, of the outer or inner unit normal, we introduce the notation $N_1^\pm\partial M$ for the corresponding unit normal bundles (where $+$ denotes outer and $-$ inner).

In the context of this paper, where we consider set-valued maps in the plane $\mathbb{R}^2$, the relevant 
unit normal bundles are isomorphic to $\mathbb{R}^2\times S^1$. The \emph{boundary map} $\beta:\mathbb R^2 \times S^1 \to \mathbb R^2 \times S^{1}$, as introduced in~\cite{kourliouros2023persistence,tey2022minimal}, 
is designed to track the outer unit normal bundle of the boundary of a set $M$ under the iterations of the set-valued map $F$ in (\ref{eq:setvaluedmap}).
\begin{definition}[Boundary Map \cite{kourliouros2023persistence}]\label{def:boundary_mapping}
    Consider the set-valued map $F$ from \eqref{eq:setvaluedmap}, induced by a diffeomorphism $f:\mathbb R^2 \to \mathbb R^2$ and $\varepsilon > 0$. Then, the boundary map $\beta:\mathbb R^2 \times S^1 \to \mathbb R^2 \times S^1$ associated to $F$ is defined as
\begin{equation}\label{eq:boundary_map}
    \beta(x,n) := \left(f(x) + \varepsilon \frac{(f'(x)^T)^{-1}n}{\|(f'(x)^T)^{-1}n\|},\frac{(f'(x)^T)^{-1}n}{\|(f'(x)^T)^{-1}n\|}\right),
\end{equation}
where we consider the natural embedding $S^{1} = \left\{n \in \mathbb R^2 : \|n\|=1\right\}$. 
\end{definition}

In Figure~\ref{fig:boundary_map}, we consider the set-valued map sketched in Figure~\ref{fig:set-valued_illustrate}, and illustrate the corresponding action of the boundary map on a boundary point $m_1 \in \partial M_1$ of the set $M_1$ with outer unit normal vector $n$. The second component of the boundary map $\beta(m_1,n)$ from (\ref{eq:boundary_map}) is the outer normal to $f(M_1)$ at the point $f(m_1)$, given by  $n_1:=\frac{(f'(m_1)^T)^{-1}n}{\|(f'(m_1)^T)^{-1}n\|}$. 
The boundary point $f(m_1)$ of $\partial f(M_1)$ relates uniquely to the boundary point $m_2:=f(m_1)+\varepsilon n_1$ on $\partial M_2$. Importantly, both boundary points share the outer normal direction. The boundary map $\beta$ thus yields $\beta(m_1,n) = (f(m_1)+\varepsilon n_1,n_1)=(m_2,n_1)$.

\begin{figure}[h]
    \centering
    \includegraphics[width=0.6\linewidth]{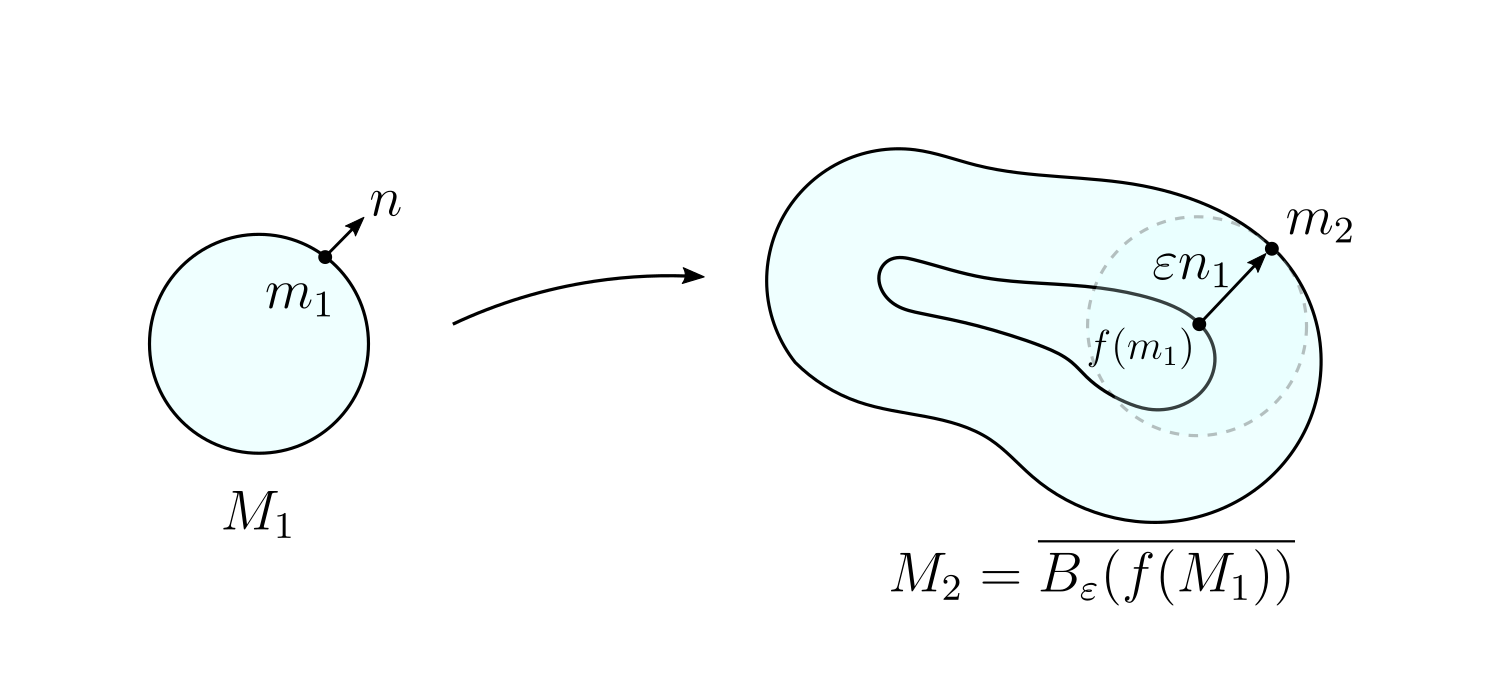}
    \captionsetup{width=\linewidth}\caption{The boundary map sends a boundary point $m_1 \in \partial M_1$ with unit normal vector $n$ to a boundary point $m_2 \in \partial M_2$ with unit normal vector $n_1 = \frac{(f'(m_1)^T)^{-1}n}{\|(f'(m_1)^T)^{-1}n\|}$.}
    \label{fig:boundary_map}
\end{figure}

Importantly, the unit normal bundle $N_1^+\partial A$ of an $F$-invariant set $A$, with continuously differentiable boundary $\partial A$, is $\beta$-invariant~\cite{kourliouros2023persistence,tey2022minimal}.
\begin{proposition}\label{prop:invariant}
    Consider an invariant set $A$ of the set-valued map $F$ with a continuously differentiable boundary $\partial A$. Then, the unit normal bundle $N_1^+\partial A$ is invariant under the boundary map $\beta$.
\end{proposition}

Remarkably, the boundary map associated to $F$ also
tracks the (inner) normal bundle of the boundary of a set $A^*$ under the action of the dual $F^*$ in (\ref{eq:dual}).
\begin{proposition}\label{prop:dual_invariant}
    Consider an invariant set $A^*$ of the set-valued map $F^*$ with a continuously differentiable boundary $\partial A^*$. Then, the unit normal bundle $N_1^-\partial A^*$ is invariant under the boundary map $\beta$.
\end{proposition}
The proof of this proposition is presented in Appendix~\ref{appb}.
Propositions~\ref{prop:invariant} and \ref{prop:dual_invariant} show that invariant sets of the boundary map $\beta$ relate to attractors and repellers, alike.

Figure~\ref{fig:set-valued_illustrate} illustrates that singularities (points of non-differentiability) can arise on the boundaries of iterates of a set with smooth boundary. Indeed, such singularities can exist also on the boundary $\partial A$ of an invariant set $A$. To characterise the occurrence of such singularities, it is useful to introduce the concept of a \emph{contributor}~\cite{lamb2020boundaries}. 
We call a point $y\in \partial f(M)$ a contributor 
of a point $x\in \partial F(M)$, if $\|x-y\|=\varepsilon$.

We illustrate the concept of contributors in Figure~\ref{fig:contribute}. We observe that singularities have more than one contributor. Indeed, it turns out that $\partial F(M)$ is differentiable at a point $x\in\partial F(M)$ if and only if $x$ has a unique contributor $y\in \partial f(M)$. 

Moreover, in the presence of singularities, not all points $y\in f(M)$ are contributors to $\partial F(M)$. Hence
\begin{equation}\label{eq:backward_invariant}
    \beta^i(N_1^+\partial M) \supset N_1^+\partial F^i(M), \;\text{for all } i \in \mathbb{N}.
\end{equation} 
In other words, the normal bundle of the smooth part of the boundary of a minimal attractor is backward invariant under $\beta$. Conversely, the normal bundle of the smooth part of its dual repeller's boundary is forward invariant under $\beta$.

\begin{figure}[h]
    \centering
    \includegraphics[width=0.8\textwidth]{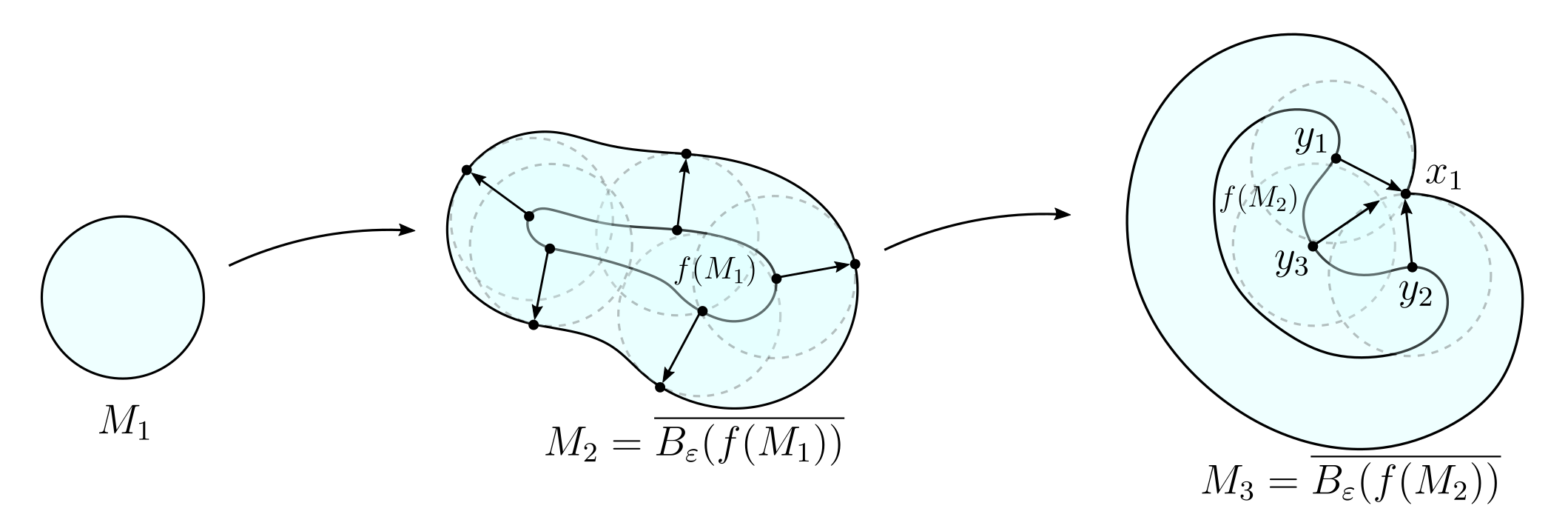}
    \captionsetup{width=\linewidth}\caption{Illustration of contributors to boundary points using the sets $M_1, M_2$ and $M_3$ in Figure~\ref{fig:set-valued_illustrate}. Each point on the smooth boundary $\partial M_2$ has a unique contributor on $\partial f(M_1)$. Conversely, the boundary point $x_1 \in \partial M_3$ has two distinct contributors $y_1,y_2 \in \partial f(M_2)$, resulting in a singular (non-differentiable) boundary point of wedge type. There are also points on $\partial f(M_2)$ which do not contribute to any boundary points on $\partial M_3$, for example, the point $y_3 \in \partial f(M_2)$.}
    \label{fig:contribute}
\end{figure}

We now consider the case where the boundary $\partial A$ of a minimal attractor $A$ is continuously differentiable. Due to the uniqueness of contributors of boundary points on $\partial A$, the boundary map $\beta$, when restricted to the normal bundle $\beta|_{N^+_1\partial A}$ is topologically conjugate to a circle homeomorphism. It is well known that a circle homeomorphism can either consist of periodic points and connections between these periodic points or represent an irrational rotation. In this paper, we focus on the situation where the normal bundle contains a finite number of periodic points of the boundary map $\beta$. It has been established that the normal bundle $N^+_1\partial A$ is normally attracting~\cite{kourliouros2023persistence}. Consequently, the normal bundle includes periodic points of $\beta$ along with the one-dimensional \emph{unstable} manifolds of saddle periodic points.

For minimal attractor $A$ with a piecewise smooth boundary $\partial A$ containing singularities, not all boundary points on the image $\partial f(A)$ contribute to $\partial A$. Here, the normal bundle $N^+_1\partial A$ remains backward invariant under the map $\beta$, forming a proper subset of the unstable manifolds associated to certain saddle periodic points.

Likewise, for the boundary $\partial A^*$ of a dual repeller $A^*$, which is continuously differentiable, its normal bundle $N^-_1\partial M^*$ is composed of periodic points of $\beta$ and the one-dimensional \emph{stable} manifolds of saddle periodic points. In the case of a dual repeller with a piecewise smooth boundary, the inward normal bundle of the smooth part of the boundary $N^-_1\partial M^*$, is forward invariant under the boundary map $\beta$.

\section{Numerical detection of the bifurcations of minimal attractors}\label{sec:numerical}

We now proceed to illustrate the use of boundary map (\ref{def:boundary_mapping}) to study numerically the boundary of minimal attractors $A$ and their dual repellers $A^*$ for the H\'{e}non map with bounded noise (\ref{eq:randomdiff}), considering various parameter regimes to showcase different types of bifurcations. 

In this section, we show that topological and boundary bifurcations of minimal attractors are related to bifurcations of periodic points of the boundary map and their associated one-dimensional invariant manifolds. Topological bifurcations are shown to be related to a fold bifurcation of periodic orbits of the boundary map in Section~\ref{sec:top_bif}, or a heteroclinic bifurcation involving two distinct fixed points in Section~\ref{sec:hetero}. Analogously for boundary bifurcations, wedge singularities are shown to arise when eigenvalues of periodic orbits change from real-valued to complex-valued in Section~\ref{sec:complex}, while the emergence of a cascade of wedge singularities is associated with a non-transversal intersection between the unstable manifold of a saddle fixed point and the strong stable foliations
of a stable periodic point, in Section~\ref{sec:cascade}. The ordering of subsections has been chosen to reflect the increasing complexity of the bifurcation: Section~\ref{sec:complex} and Section~\ref{sec:top_bif} delve into bifurcations detectable through monitoring the eigenvalues at periodic points, while Section~\ref{sec:hetero} and Section~\ref{sec:cascade} concerns global bifurcation in the boundary map. The figures presented in this paper can be reproduced using the source codes in~\cite{Tey2024}. The set-oriented numerical method is implemented in GAIO~\cite{dellnitz2001algorithms} and the approximation of one-dimensional unstable and stable manifolds follows standard techniques (cf.~\cite{you1991calculating}).

\begin{figure}[!b]
    \begin{subfigure}[t]{0.48\textwidth}
    \centering
    \includegraphics[width=1\textwidth]{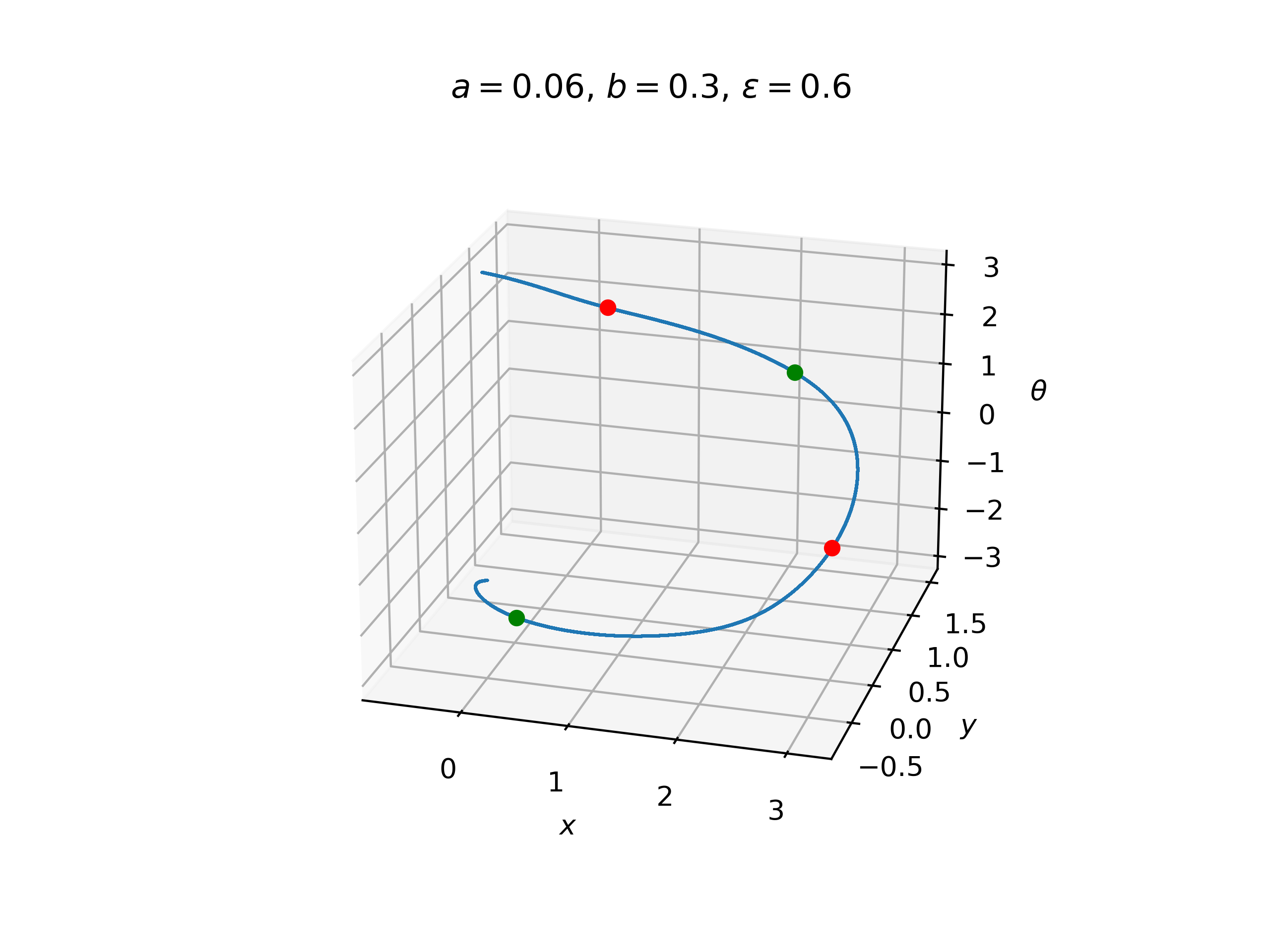}
    \captionsetup{width=\linewidth}\caption{Normal bundle of the boundary of the minimal attractor, formed by a heteroclinic cycle of the boundary map. Unit normal vectors $n=(n_1,n_2)^T\in S^1$ 
are represented by the angle $\theta := \arctan(n_2/n_1) \in (-\pi, \pi]$. }
    \label{fig:3d_a006_eps06}
    \end{subfigure}
    \hfill
    \begin{subfigure}[t]{0.48\textwidth}
    \centering
    \includegraphics[width=1\textwidth]{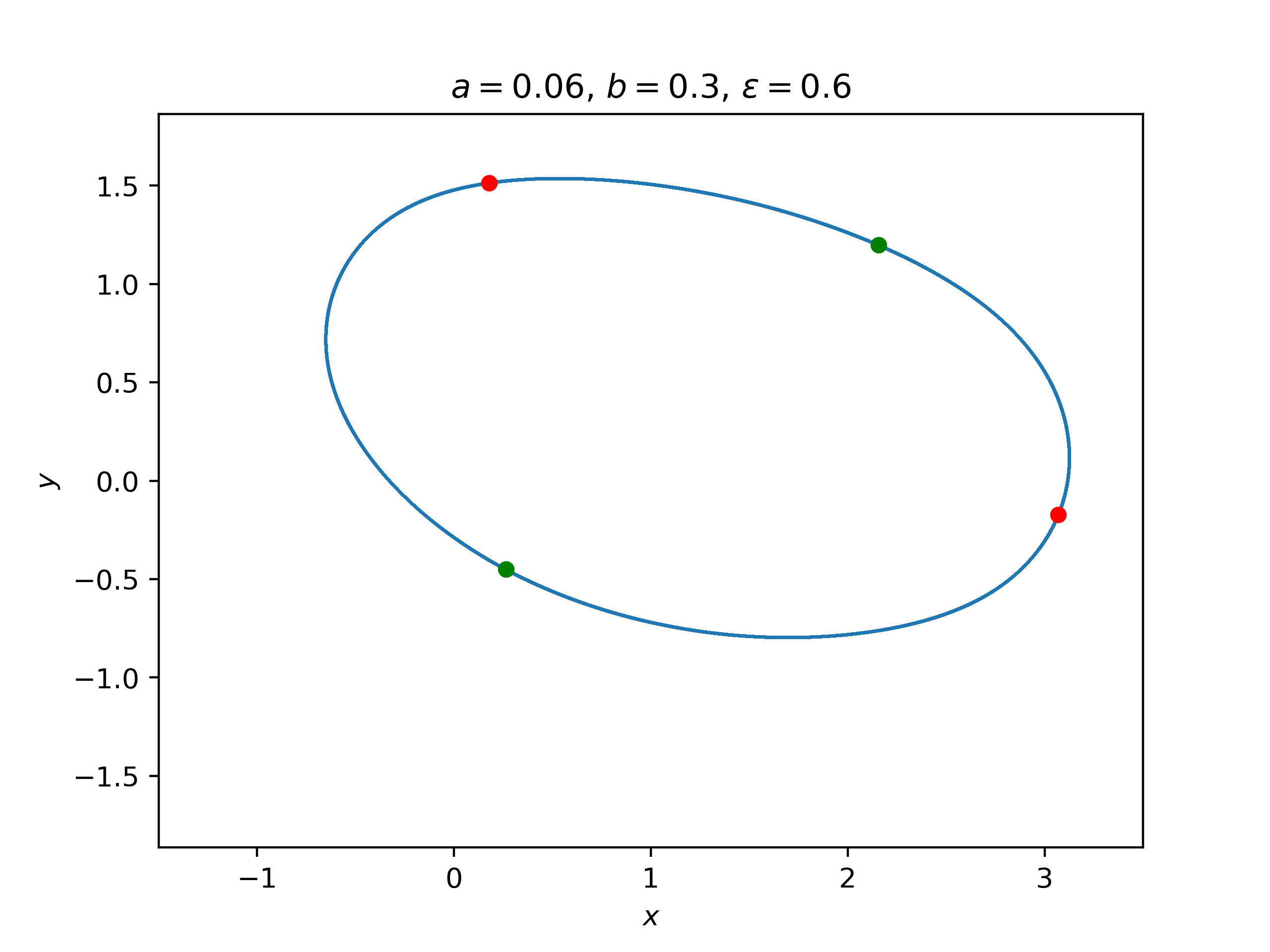}
    \captionsetup{width=\linewidth}\caption{Boundary of the minimal attractor: the orthogonal projection of the heteroclinic cycle in (a).}
    \label{fig:a006_eps06_2d}
    \end{subfigure}
    \captionsetup{width=\linewidth}\caption{Numerical approximation of the normal bundle of the boundary of the minimal attractor of the  H\'{e}non map with bounded noise \eqref{eq:randomdiff} with parameter values $a = 0.06, b = 0.3$ and $\varepsilon = 0.6$. The normal bundle is a heteroclinic cycle of the boundary map $\beta$ \eqref{eq:boundary_map}, consisting of a saddle periodic orbit of period two (red dots), two stable fixed points (green dots) and saddle connections between them (the unstable manifolds of the saddle periodic points), see (a).  The projection of this cycle on the $x$-$y$ plane in (b) yields the boundary of the minimal attractor. 
    All fixed and periodic points in the heteroclinic cycle have real eigenvalues and the boundary is smooth.
    }
    \label{fig:eps06_a006}
\end{figure}

\begin{figure}[h!]
    \begin{overpic}[width=.52\textwidth,trim={0 0 0 25},clip]{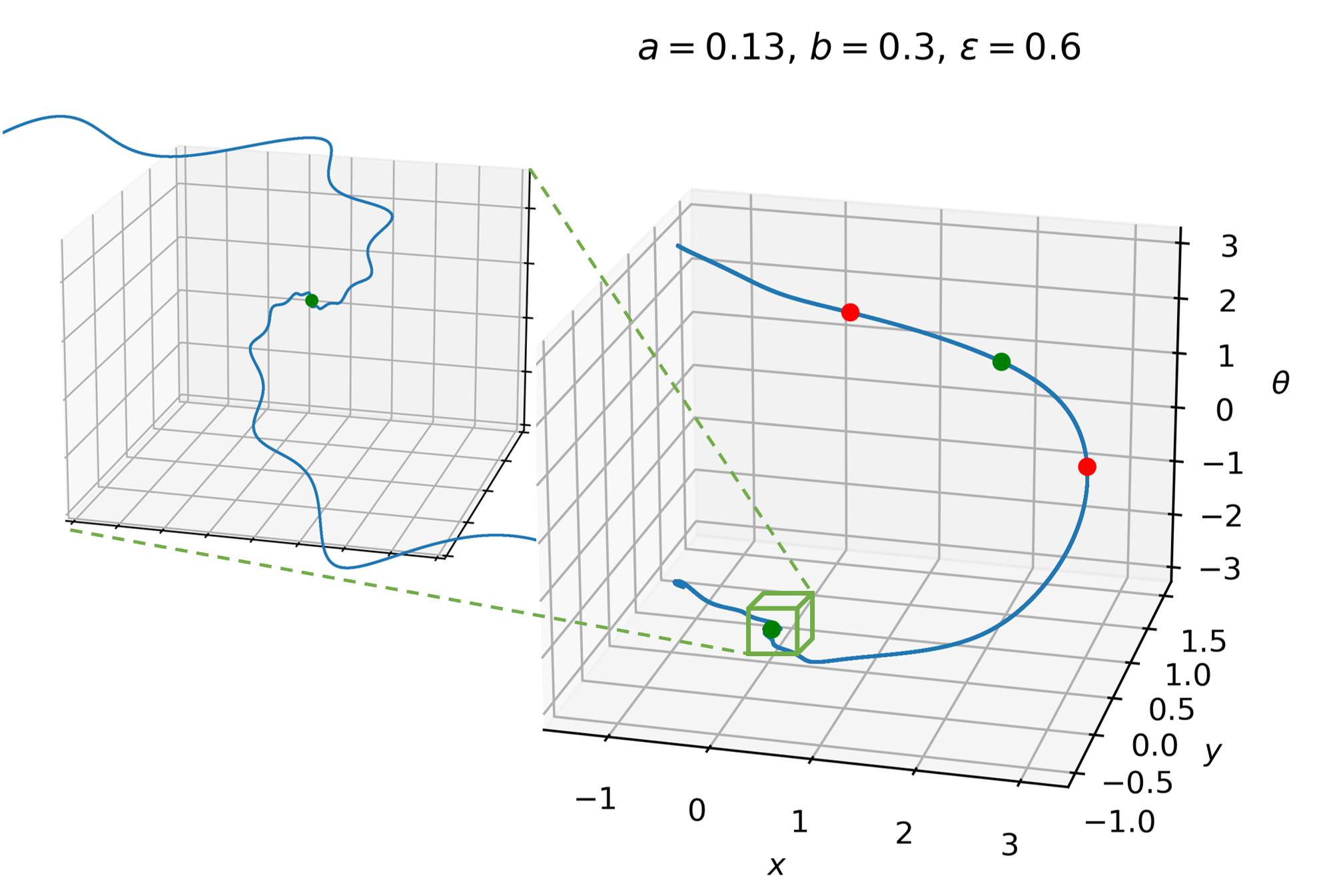}
    \put(0,0){
    (a)
    }
    \end{overpic}
    \begin{overpic}[width=.46\textwidth,trim={0 0 0 40},clip]{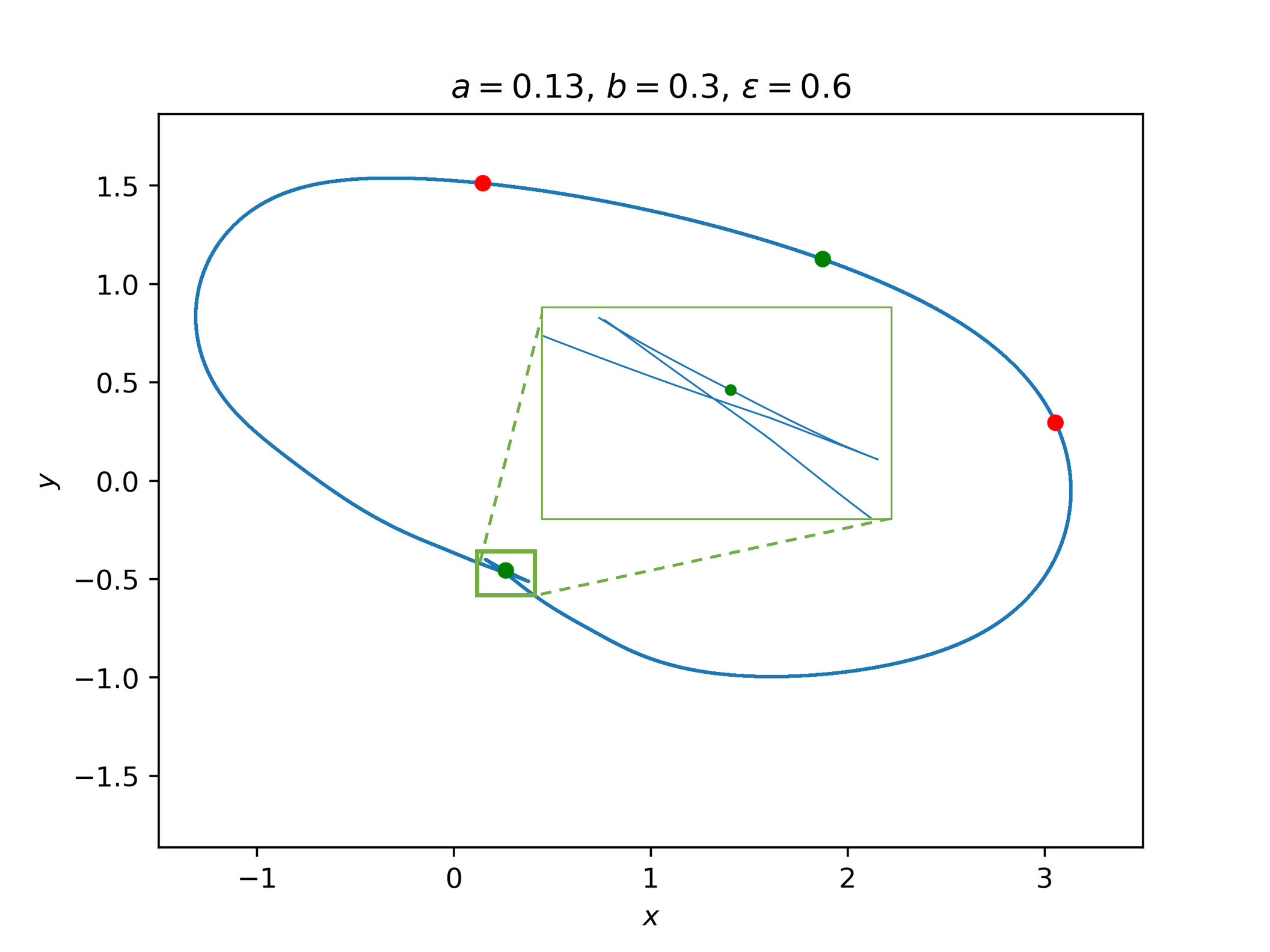}
    \put(0,0){
    (b)
    }
    \end{overpic}

    \captionsetup{width=\linewidth}\caption{Boundary of the minimal attractor of the H\'{e}non map with bounded noise \eqref{eq:randomdiff} with parameter values $a=0.13$, $b=0.3$ and $\varepsilon = 0.6$, displaying a wedge singularity.
    On the left, in (a), a numerical approximation of the heteroclinic cycle of the boundary map is depicted, representing the normal bundle of the boundary.
    The Jacobian of the boundary map at the fixed point $(0.261,-0.455,-2.046)$ on this heteroclinic cycle has eigenvalues $\{0.515,-0.685+0.215i,-0.685-0.215i\}$. These complex eigenvalues cause a spiralling of the unstable manifolds that approach this fixed point in the heteroclinic cycle. The projection of these spiralling manifolds gives rise to a wedge singularity, see the magnified inset in part (a). The projection of the heteroclinic cycle to the state space is depicted in part (b), with the inset showing a magnification of the area with the wedge.
    }
    \label{fig:eps06_a01_013}
\end{figure}

\subsection{Creation of wedge singularity through complex eigenvalue splitting for periodic orbits of the boundary map}\label{sec:complex}
We consider the birth of an isolated singular (wedge) point on a smooth (part of the) boundary. This boundary bifurcation arises at a fixed point of the boundary map when the eigenvalues of the linearised boundary map change from real to complex.

We consider the H\'{e}non map with bounded noise (\ref{eq:randomdiff}) with parameter values $b = 0.3$, $a = 0.06$, and $\varepsilon = 0.6$, which we find to have a unique minimal attractor with smooth boundary with a normal bundle that is a heteroclinic cycle of the boundary map $\beta$ \eqref{eq:boundary_map} between a two-periodic saddle orbit and two stable fixed points, see Figure~\ref{fig:3d_a006_eps06}. The saddle points have one-dimensional unstable manifolds that connect to the stable fixed points.  
The
orthogonal projection of the heteroclinic cycle in Figure~\ref{fig:3d_a006_eps06} forms the boundary of the minimal attractor, see Figure~\ref{fig:a006_eps06_2d}.

The eigenvalues
of the Jacobian at the stable fixed point with coordinates $(0.263,-0.451,$ $-2.060)$ are
$\{0.532,-0.768,-0.693\}$. We approach a boundary bifurcation by gradually increasing the value of the parameter $a$ from $a=0.06$ and follow the continuation of the heteroclinic cycle whose projection yields the minimal attractor. The eigenvalues of the Jacobian of the continuation of the fixed point, mentioned above,   remain real, until at $a \approx 0.06191$, the two negative eigenvalues collide on the real axis, after which they branch off into the complex plane (as a complex conjugate pair). We observe that the change of stability type of the fixed point, which induces a spiralling behaviour of the connecting unstable manifold to the stable fixed point, creates a wedge singularity on the boundary of the attractor.
The situation is sketched in detail in Figure~\ref{fig:eps06_a01_013} at $a=0.13$. 

We note that the non-smoothness of the boundary at the wedge singularity corresponds to a point of self-intersection of the projection of the heteroclinic cycle of the boundary map. Indeed, the smoothness of the boundary is associated with the absence of such intersections.

It turns out that the form of the boundary map gives rise to a special relationship between the eigenvectors and eigenvalues of the Jacobian at a fixed point of the boundary map when the Jacobian has complex eigenvalues. We observe from direct calculations (see Proposition~\ref{prop:eigenrelation} in Appendix~\ref{sec:eigen}) that there is a unique splitting into a direct sum of a one-dimensional and a two-dimensional subspace of the (three-dimensional) tangent space that are invariant under the Jacobian, where the two-dimensional space (naturally) projects to the (tangent space of the) state space $\mathbb{R}^2$ as a line. 
Moreover, if $\lambda_1$ denotes the eigenvalue corresponding to the one-dimensional invariant subspace and $\lambda_2,\lambda_3$ are the complex eigenvalues then
$\lambda_1=\lambda_2\lambda_3$. 
Consequently, there is a spectral gap between these invariant subspaces so that there exists a unique two-dimensional weak stable manifold (whose tangent space is equal to the two-dimensional invariant subspace, mentioned above) and every orbit of the boundary map that starts near the fixed point approaches the fixed point along this weak stable manifold.

The birth of an isolated wedge singularity can be understood, already, from the linearisation of the boundary map near a stable fixed point. For instance, in Figure~\ref{fig:linear} we sketch the situation as arises in the linearisation of the boundary map and its relevant projection, for the H\'enon map example at parameter values $a = 0.06$ (real eigenvalues) and $a = 0.13$ (pair of complex conjugate eigenvalues).

\begin{figure}[t]
    \begin{overpic}[width=.48\textwidth,trim={0 0 0 0},clip]{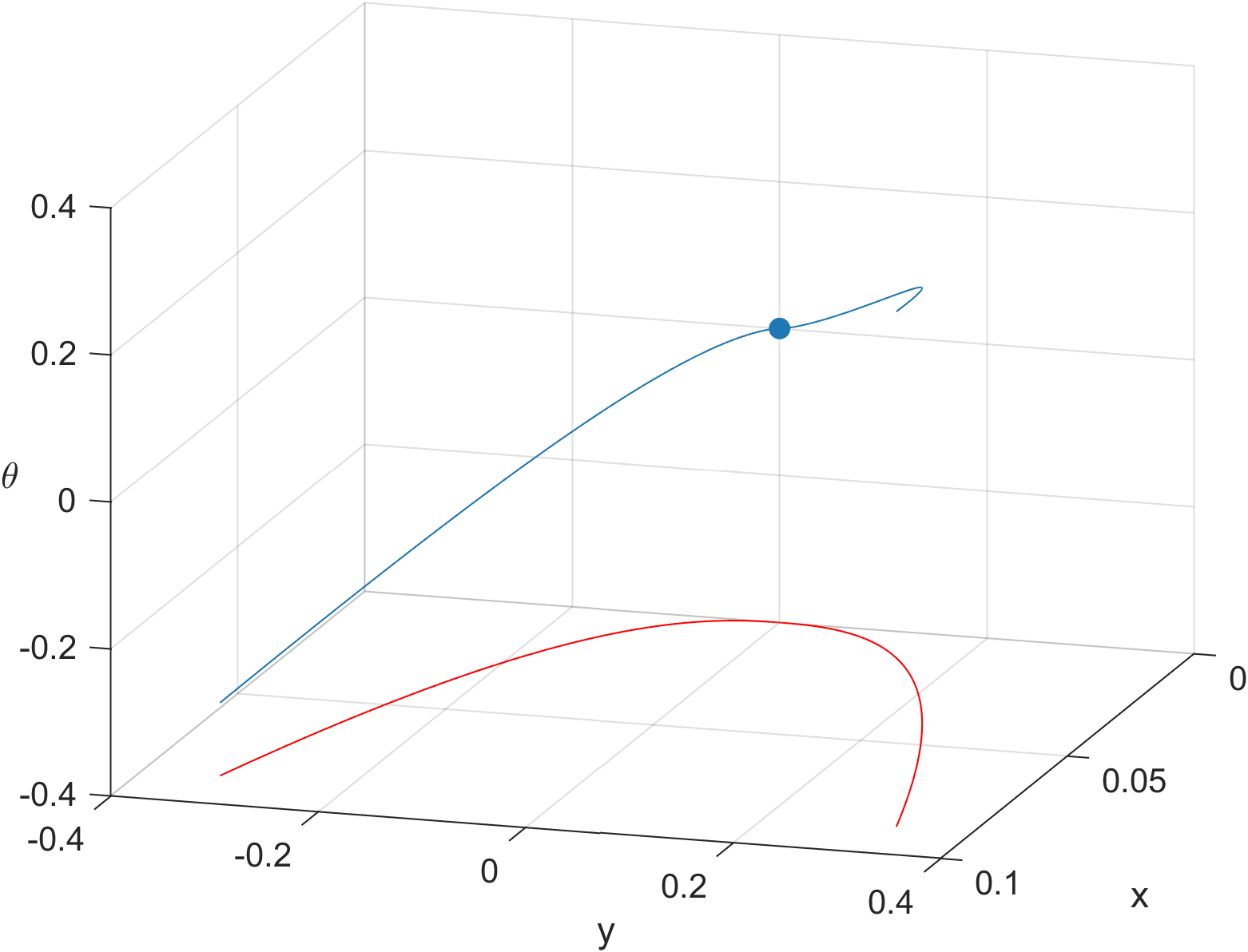}
    \put(0,0){
    (a)
    }
    \end{overpic}
    \begin{overpic}[width=.48\textwidth,trim={0 0 0 0},clip]{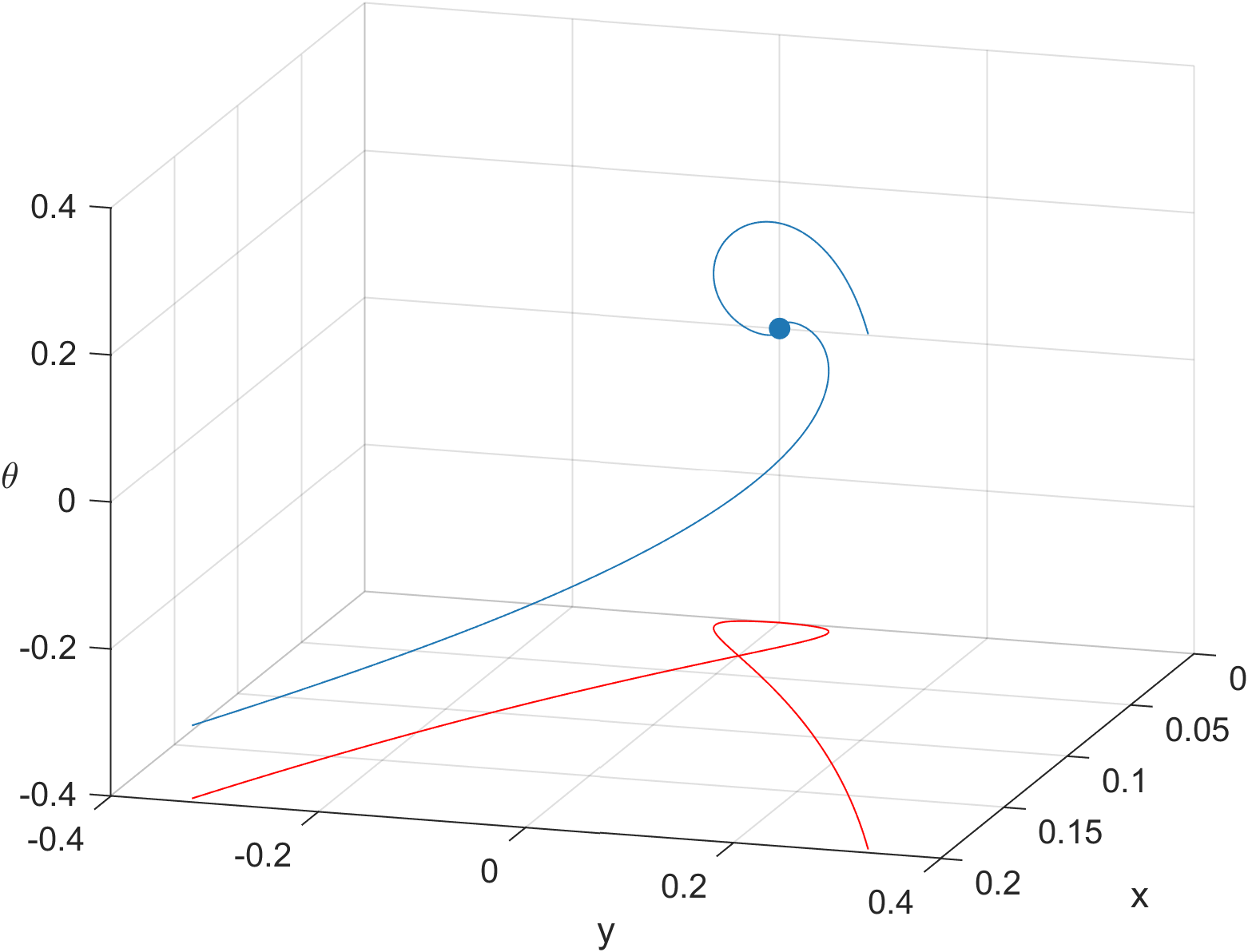}
    \put(0,0){
    (b)
    }
    \end{overpic}
    \captionsetup{width=\linewidth}\caption{Sketch of the birth of an isolated wedge singularity in relation to boundary map dynamics. We model the linearised dynamics near a stable fixed point of the boundary map as found in the H\'{e}non map, as in Figure~\ref{fig:eps06_a01_013}, with (a) $a=0.06$ (real eigenvalues) and (b) $a = 0.13$ (complex eigenvalues). We consider the linearised dynamics in Jordan normal form, to simplify the presentation: $\left(\begin{smallmatrix}
    0.532&0&0\\
    0&-0.768&0\\
    0&0&-0.693
    \end{smallmatrix}\right)$ in (a) and $\left(\begin{smallmatrix}
    0.515&0&0\\
    0&-0.685&0.215\\
    0&-0.215&-0.685
    \end{smallmatrix}\right)$ in (b). Blue curves represent smooth invariant manifolds approaching the stable fixed point and their relevant projections are presented in red. The transition from real to complex eigenvalues coincides with the birth of the wedge singularity.}
    \label{fig:linear}
\end{figure}

The properties of eigenvalues and eigenvectors mentioned above, give rise to specific properties of the relevant projection.
In the case of real eigenvalues, smooth invariant curves approaching the stable fixed point from opposite directions, project to a smooth curve, see Figure~\ref{fig:linear}(a). In the case of complex eigenvalues, the invariant subspace for the Jacobian corresponding to the complex eigenvalues projects to a line. 
 Then, smooth invariant curves spiralling toward the fixed point from opposite directions outside of this subspace, give rise to a cascade of transverse intersection points in the projection, the first one creating a wedge singularity as seen in Figure~\ref{fig:linear}(b). The rest of the intersections are too close to the origin, to be visible in this figure.
 
This example shows how a change of stability type of a stable fixed point of the boundary map on the normal bundle of the minimal attractor from real to complex eigenvalues leads to the birth of an isolated wedge singularity. Indeed, only the first intersection, which creates the wedge, is relevant to the boundary, as the remainder projects to the interior of the attractor, cf. Figure~\ref{fig:eps06_a01_013}(b). We conjecture that the birth of singularity following a stability type change of the eigenvalues from real to complex generically applies. A mathematical proof, taking into account the relevant geometric setting of the boundary map, requires further formalisation and is beyond the scope of this paper.

\begin{figure}[!b]
\begin{minipage}{.007\textwidth}
    \text{}
\end{minipage}
\begin{overpic}[width=.47\textwidth,trim={0 0 0 0},clip]{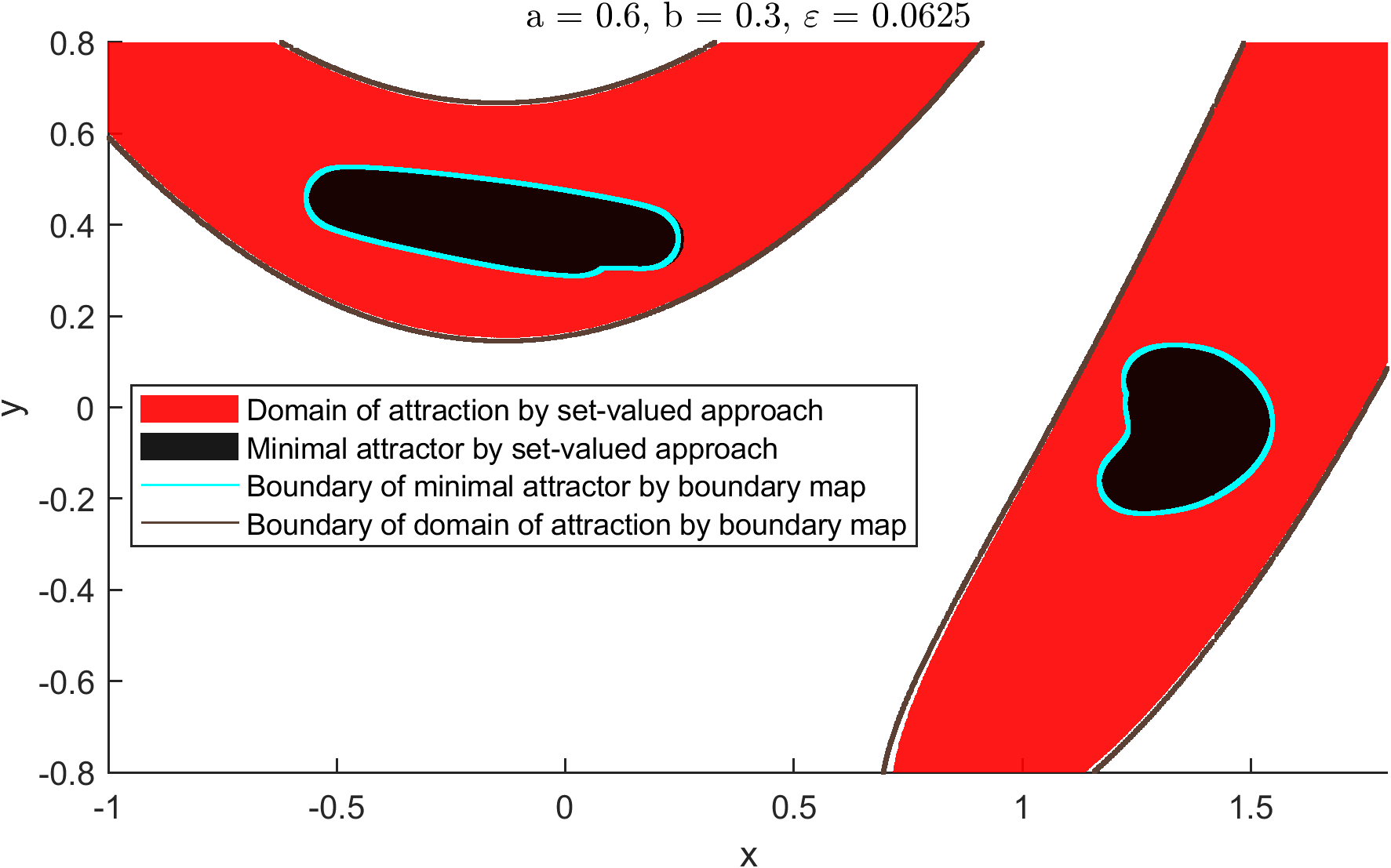}
\put(-8,0){
(a)
}
\end{overpic}
\begin{minipage}{.008\textwidth}
    \text{}
\end{minipage}
\hfill
\begin{overpic}[width=.47\textwidth]{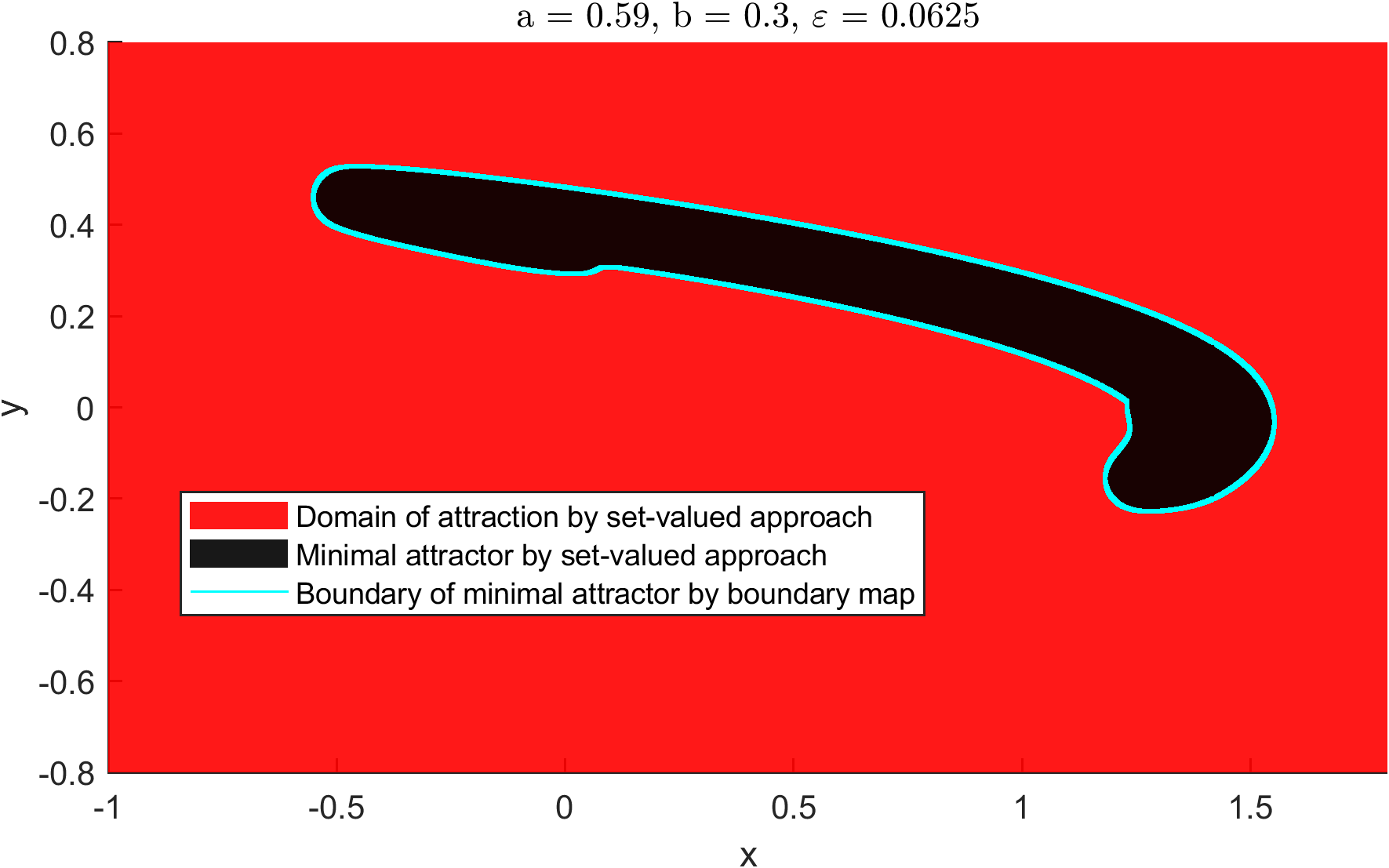}
\put(-5,0){
(b)
}
\end{overpic}
\begin{minipage}{.007\textwidth}
    \text{}
\end{minipage}
\vspace{.3em}

\begin{overpic}[width=.49\textwidth,trim={0 0 0 1},clip]{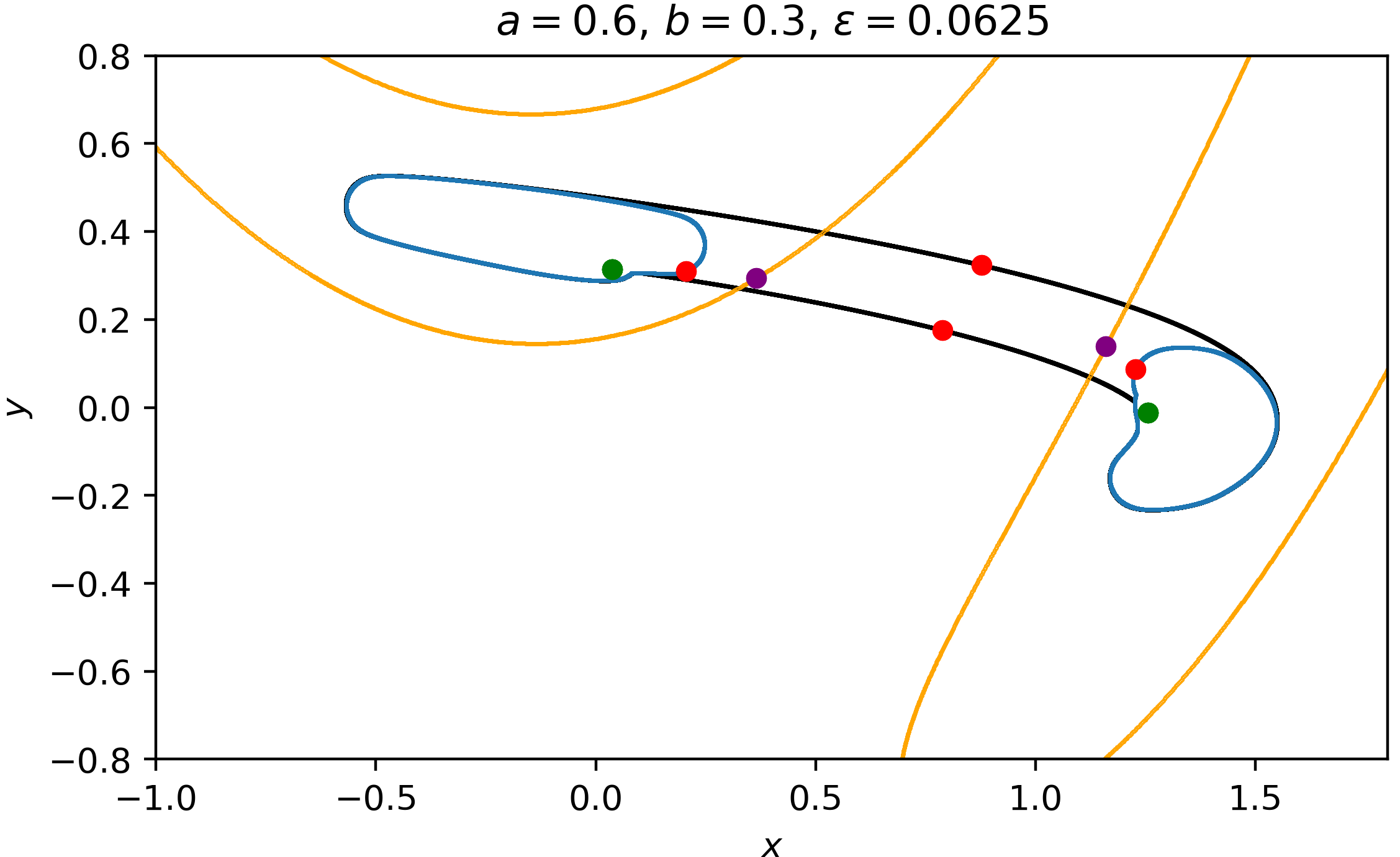}
\put(0,0){
(c)
}
\end{overpic}
\begin{overpic}[width=.49\textwidth,trim={0 0 0 0},clip]{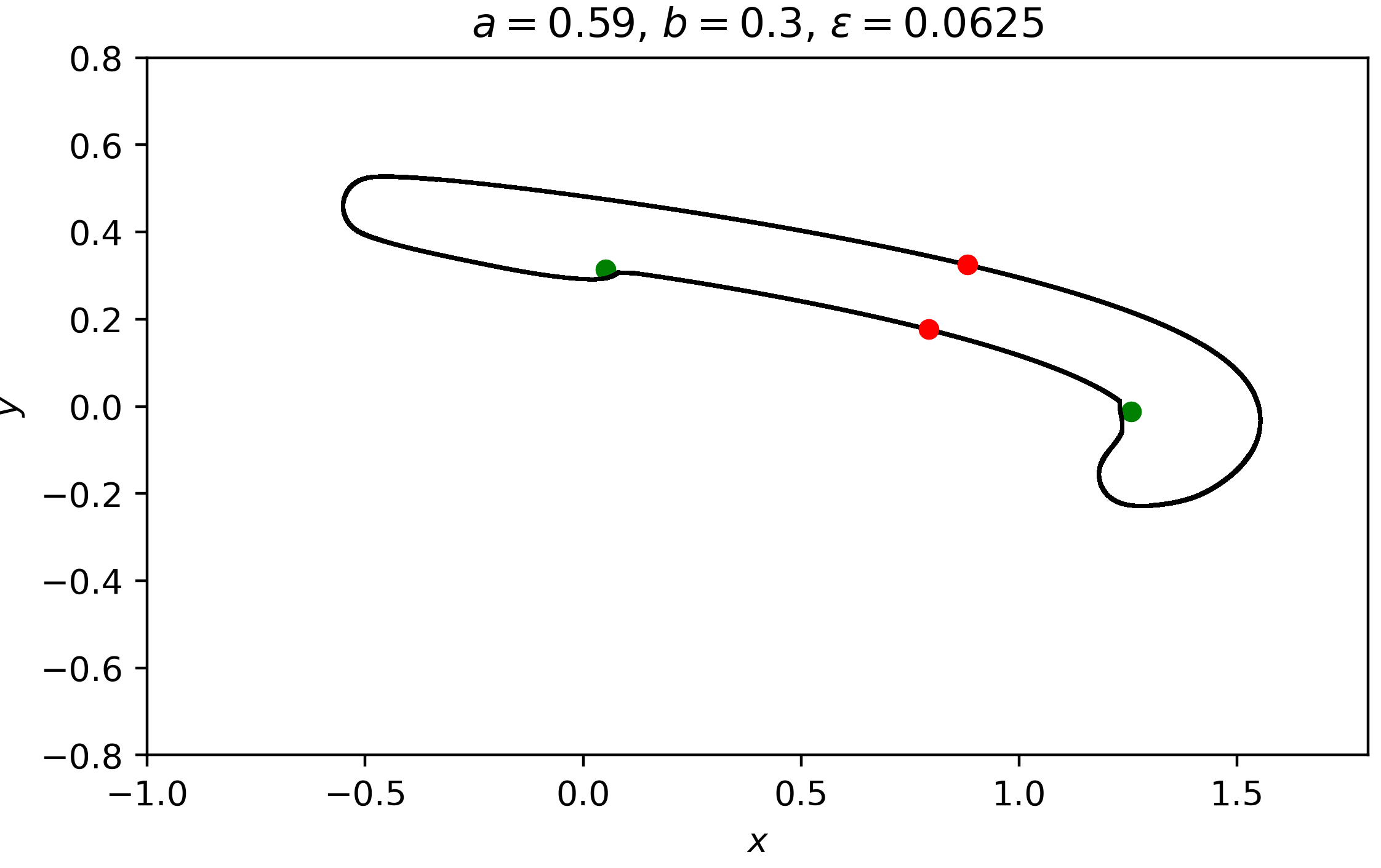}
\put(0,0){
(d)
}
\end{overpic}

    \captionsetup{width=\linewidth}\caption{Numerical approximation of the minimal attractor of the H\'{e}non map with bounded noise (\ref{eq:randomdiff}) and its domain of attraction (using  GAIO) for the parameter values $a=0.6$ in (a), and $a= 0.59$ in (b), with $b = 0.3, \varepsilon = 0.0625$ in both cases. When $a \approx 0.595$, a discontinuous topological bifurcation arises where the minimal attractor consisting of two disjoint components collides with its dual repeller (white region) and bifurcates into a single connected attractor. From the viewpoint of the boundary map, the boundaries of minimal attractors and domain of attraction are (parts of) projections of certain one-dimensional invariant manifolds associated to fixed points and two-periodic points of saddle type (blue in (c) and black in (d)). The topological bifurcation of the set-valued map is linked to a fold bifurcation of the boundary map where two two-periodic orbits of saddle type collide and disappear, together with their associated invariant manifolds. See the main text for a more detailed discussion. 
    }
    \label{fig:a0.6}
\end{figure}

\subsection{Topological bifurcation through a fold bifurcation of the boundary map}\label{sec:top_bif}

We consider a topological bifurcation of minimal attractors of the H\'{e}non map with bounded noise (\ref{eq:randomdiff}), corresponding to a fold bifurcation of two-periodic saddle orbits of the boundary map, as is observed around $a = 0.595$, $b = 0.3$ and $\varepsilon = 0.0625$.

In Figure~\ref{fig:a0.6}, we present numerical approximations of the minimal attractors when (a) $a=0.6$ and (b) $a = 0.59$. The minimal attractor is depicted in black and its domain of attraction in red. When $a=0.6$, the attractor consists of two disjoint components that are permuted by the dynamics, whereas when $a=0.59$, there is one connected attractor. The transition between these different situations arises discontinuously \cite{lamb2015topological}: when decreasing $a$ from $a=0.6$ to $a=0.59$, around $a=0.595$ the two-component attractor explodes into one connected attractor, in a lower semi-continuous manner \cite{lamb2015topological}. The coloured regions reflect the results of (brute-force) computations using the set-oriented numerical toolbox GAIO \cite{dellnitz2001algorithms}, which is well-suited for this task.

The boundaries of attractors and their domains of attraction can also be computed using the boundary map,
as they are formed by (parts of) invariant manifolds of this map.\footnote{Recall that the boundary of the domain of attraction of an attractor is equal to the boundary of its dual repeller (attractor of the dual set-valued map), and the boundary map for the set-valued map and its dual are each other's inverses, cf.~Section~\ref{sec:bm}. }
At $a=0.6$, we find the boundary of the minimal attractor of the set-valued map to be related to a heteroclinic cycle between a two-periodic saddle orbit and an attracting two-periodic orbit of the boundary map: the saddle periodic points have one-dimensional unstable manifolds (in the three-dimensional state space of the boundary map), which connect to the attracting two-periodic points. In Figure~\ref{fig:a0.6}(c) we present a projection of the relevant dynamical objects of the three-dimensional boundary map to the two-dimensional state space. The attracting two-periodic orbit is represented by two green dots and the saddle-periodic orbit by two red dots (nearby). The blue curves are projections of (relevant parts of) the unstable manifold that make up the boundary of the minimal attractor.
The boundary of the domain of attraction is represented in orange and arises from the projection of one-dimensional stable manifolds of another two-periodic saddle orbit of the boundary map, represented by two purple dots.
These correspondences are readily verified from Figures~\ref{fig:a0.6}~(a) and (c). Indeed, we produced Figure~\ref{fig:a0.6}(a) by superimposing the orange and blue curves onto the results of the computations with GAIO, where we find a good match and note that the results from the boundary map are more accurate with far less computational effort.

We now proceed to discuss some other objects included in Figure ~\ref{fig:a0.6}(c) which are relevant for the understanding of the dynamics and the pending topological bifurcation. There are two additional red points, outside of the domain of attraction, which are of saddle type with one-dimensional unstable manifolds. These manifolds also connect to the attracting two-periodic points represented by the green dots.
In fact, the projections of (parts of) these manifolds, represented by black curves partly obscured by the blue curves, form the boundary of an attractor which is not minimal: indeed, it contains the minimal attractor with two disjoint components bounded by the blue curves.

The topological bifurcation arises as, with decreasing $a$, the red saddle two-periodic orbit on the blue curve approaches the purple saddle two-periodic orbit on the orange curve, leading to a fold bifurcation in which these two two-periodic orbits collide and disappear. This coincides, for the set-valued dynamics, to a collision of the minimal attractor with its dual repeller. This destroys the two-component minimal attractor, which involved one of the saddle two-periodic orbits. However, the non-minimal attractor remains to exist and becomes minimal through the disappearance of the other attractor within it. We depict the projection of (parts of) the relevant manifolds contributing to the boundary of the remaining attractor at $a=0.59$ in Figure~\ref{fig:a0.6}(d) in black.  The correspondence with the numerical results for the minimal attractor with GAIO, as presented in Figure~\ref{fig:a0.6}(b), is readily verified. 

We would like to remark that the complexity of the (projection of) relevant unstable manifolds of the boundary map may complicate the understanding of the boundary of the minimal attractor. Indeed,  in Figure~\ref{fig:a0.6}(c) and (d), parts of the projection of the heteroclinic cycle that is deemed irrelevant for the boundary, lying in the interior of the attractor were omitted. In Figure~\ref{fig:06_nocut}, we show the entire projection of the heteroclinic cycle between the two-periodic saddle orbit and the two-periodic stable orbit of the boundary map at $a=0.6$, illustrating the practical complication of deciding which parts of the manifold contribute to the boundary. While this can be resolved satisfactorily in a numerical way, yielding a good approximation of the boundary, the complexity of the underlying (chaotic) dynamics of the boundary map makes it difficult to obtain the precise structure of the boundary. However, a better understanding of the underlying dynamics may elucidate generic qualitative features.

\begin{figure}[!tbp]
    \centering
    \includegraphics[width=0.95\linewidth]{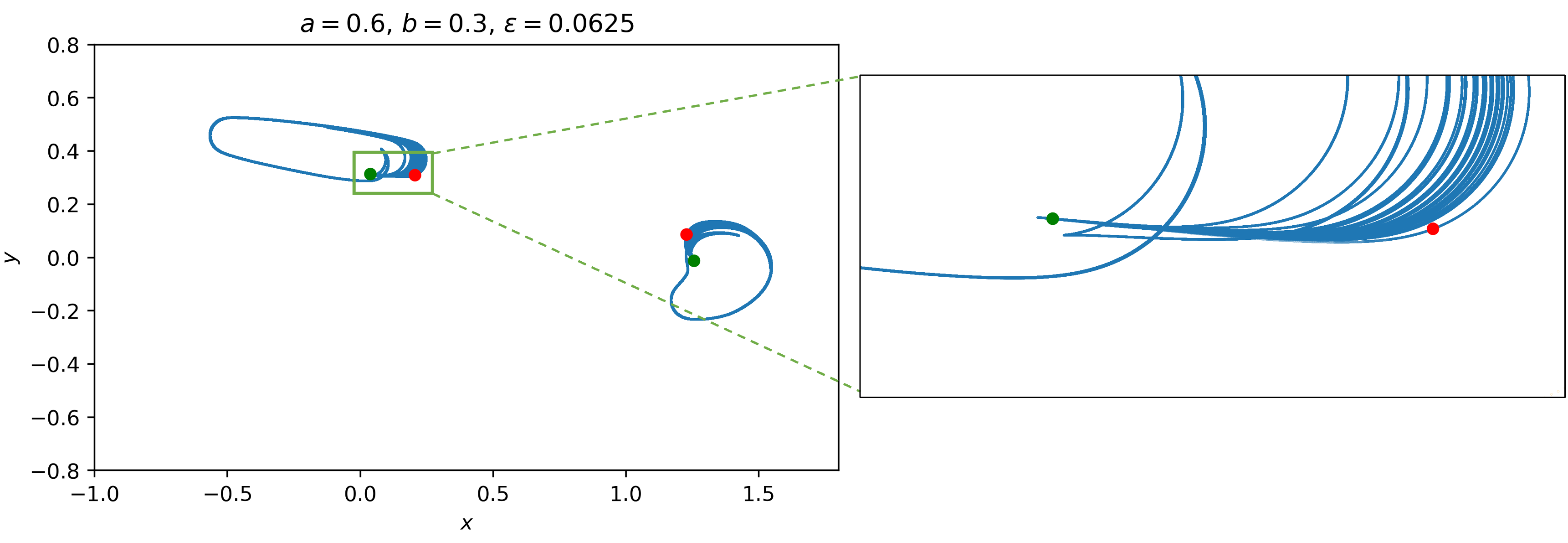}
    \captionsetup{width=\linewidth}\caption{The projection of the heteroclinic cycle between a saddle two-periodic orbit (red dots) and a stable two-periodic orbit (green dots) of the boundary map for the H\'{e}non map. Only part of the projection corresponds to the boundary of the minimal attractor and the remaining projected manifolds lie in the interior of the minimal attractor. The boundary of the minimal attractor is obtained by removing parts of the projected manifolds in the interior, resulting in the blue curves in Figure~\ref{fig:a0.6}(c). 
    We observe four wedge singularities on the boundary, two on each disjoint components, see the magnified inset. We note the complexity of the invariant manifolds, complicating the analysis.}
    \label{fig:06_nocut}
\end{figure}

This bifurcation is best appreciated in movie format, tracing superpositions of GAIO data and relevant dynamical objects from the boundary map as $a$ varies. Such movies for this example are provided in \cite{Teyvideo2024} for $a = 0.607$ to $a=0.003$, and in \cite{Teyvideoslow2024} for $a=0.607$ to $a=0.58$ in a slower motion, focusing on topological bifurcation.

\subsection{Topological bifurcation due to a heteroclinic bifurcation of the boundary map}\label{sec:hetero}

Consider the H\'{e}non map with bounded noise (\ref{eq:randomdiff}) at parameter values $b = 0.3$ and $\varepsilon = 0.6$ and $a=0.36$. This random dynamical system has a unique minimal attractor, and it is displayed together with its domain of attraction in Figure~\ref{fig:set-valued_eps06}(a). When $a$ is increased, the attractor moves towards its dual repeller (cf.~Figure~\ref{fig:set-valued_eps06}(b)), and after it collides, when $a\approx 0.37427117$, the attractor disappears instantaneously. From the set-valued point of view, this scenario is similar to that of the example in Section~\ref{sec:top_bif}, above, cf.~also \cite{lamb2015topological}.

\begin{figure}[!htbp]
    \begin{overpic}[width=0.49\textwidth,trim={0 0 0 0},clip]{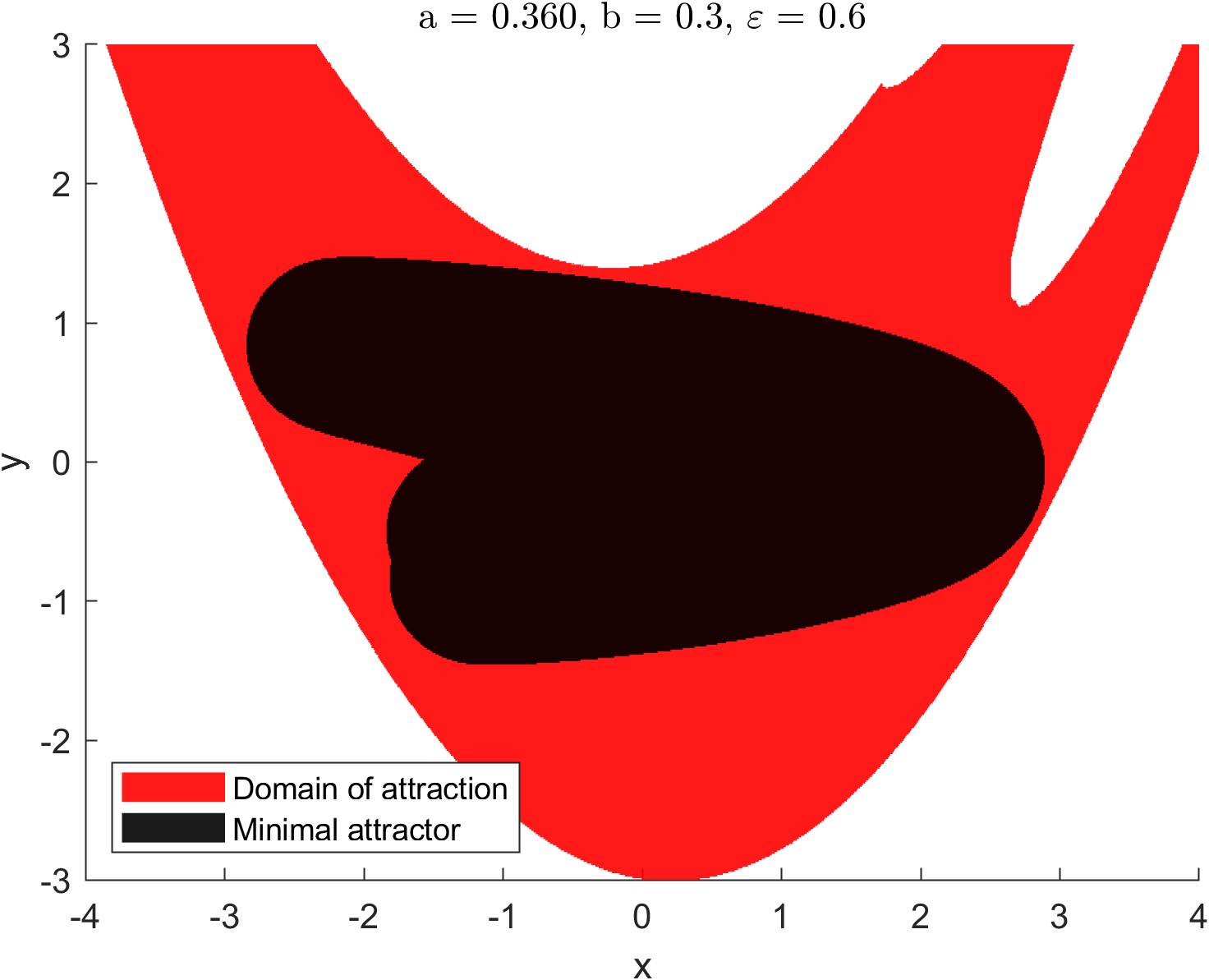}
    
    \put(0,0){
    (a)
    }
    \end{overpic}
    \begin{overpic}[width=.49\textwidth,trim={0 0 0 0},clip]{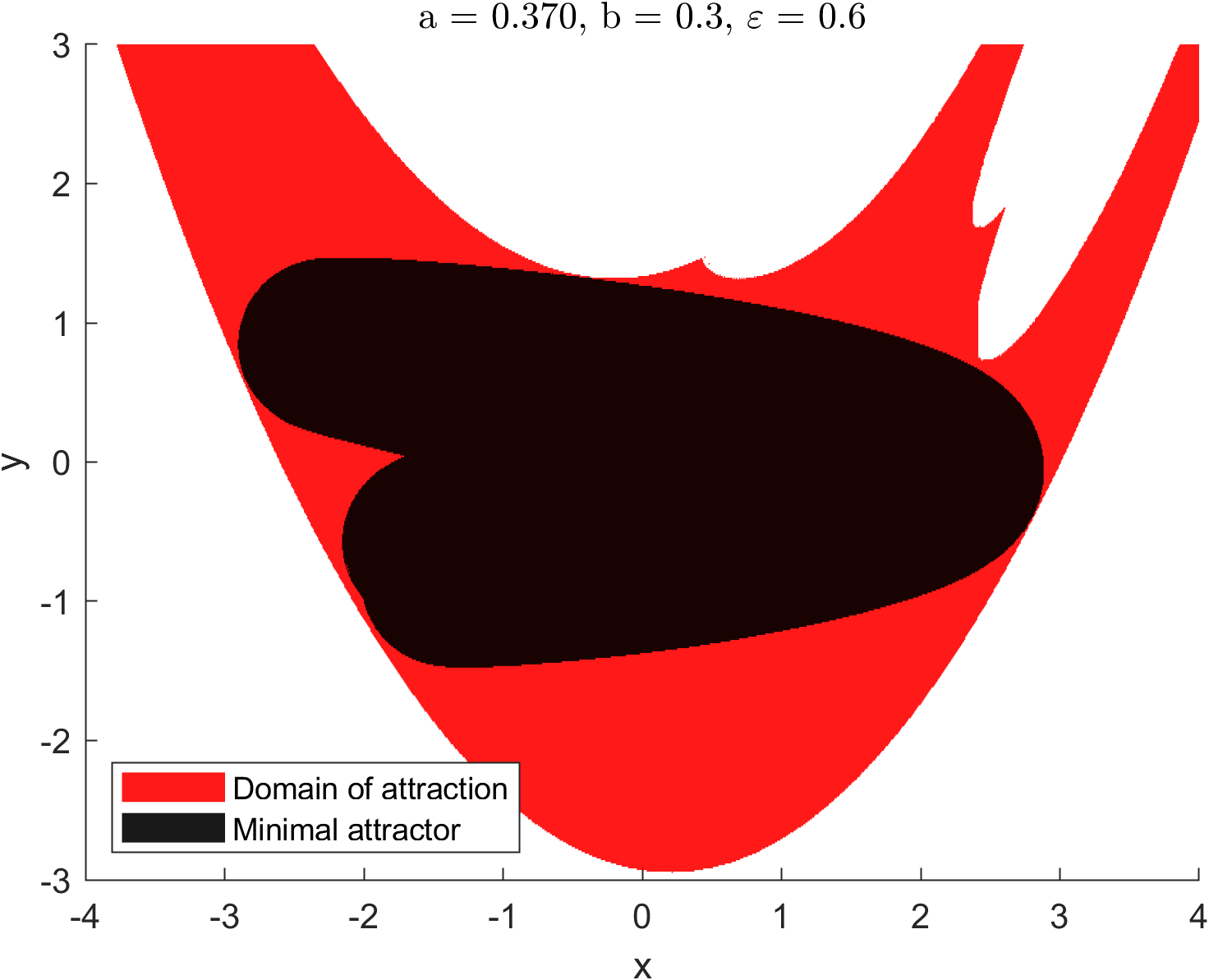}
    
    \put(0,0){
    (b)
    }
    \end{overpic}
    \captionsetup{width=\linewidth}\caption{Minimal attractor of the H\'{e}non map with bounded noise (\ref{eq:randomdiff}) and its domain of attraction, computed using GAIO, at parameter values $b = 0.3$, $\varepsilon = 0.6$, and (a) $a = 0.36$, (b) $a = 0.37$.
    As the parameter values $a$ increases from $a = 0.36$ to $a = 0.37$, the minimal attractor moves towards its dual repeller (white region). If $a$ is further increased, a topological bifurcation occurs at $a \approx 0.37427117$ where the attractor collides with the repeller and the attractor disappears instantaneously.}
    \label{fig:set-valued_eps06}

    \begin{overpic}[width=0.49\textwidth,trim={0 0 0 25},clip]{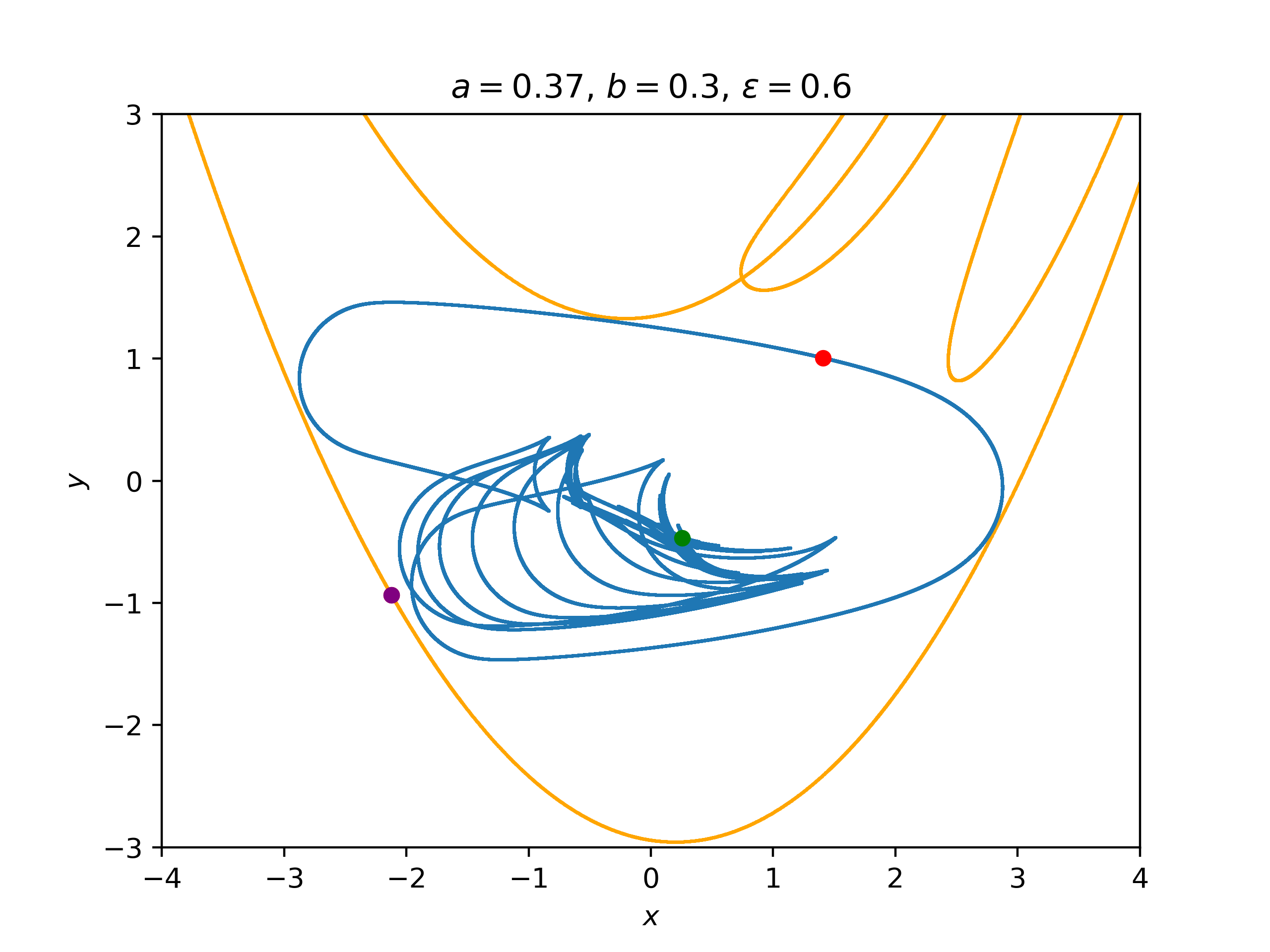}
    
    \put(0,0){
    (a)
    }
    \end{overpic}
    \begin{overpic}[width=.49\textwidth,trim={0 0 0 25},clip]{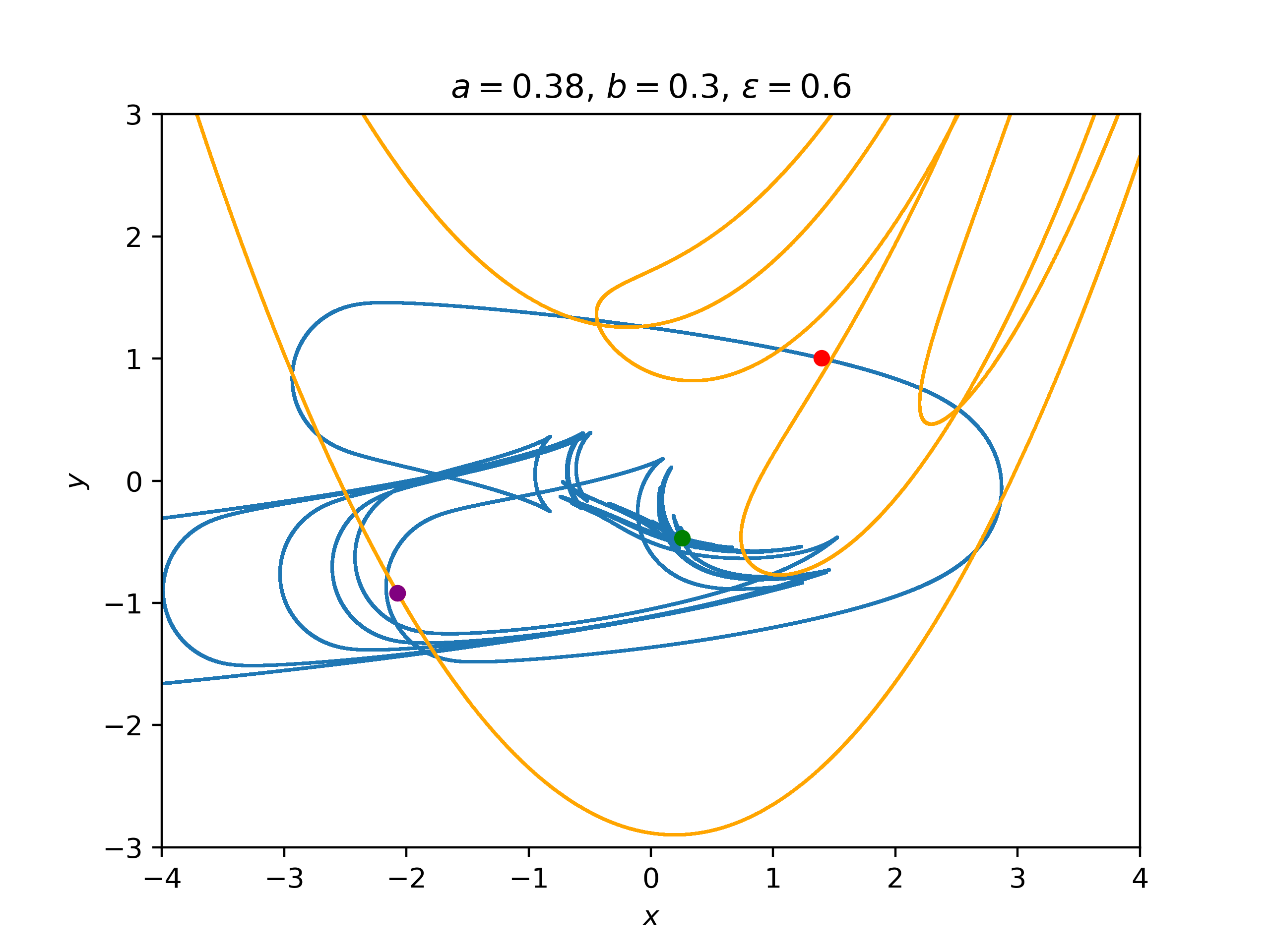}
    
    \put(0,0){
    (b)
    }
    \end{overpic}

    \captionsetup{width=\linewidth}\caption{Relevant projections of parts of invariant manifolds involved in the topological bifurcation of the 
    minimal attractor of the H\'{e}non map with bounded noise (\ref{eq:randomdiff}) at parameter values $a \approx 0.37427117$, $b = 0.3$, and $\varepsilon = 0.6$:
    the unstable manifold (blue curve) of a saddle fixed point (red dot) for the boundary of the attractor and the stable manifold (orange curve) of another saddle fixed point (purple dot) for the boundary of its dual repeller.
    In (a), at $a = 0.37$, we observe a situation close to a heteroclinic connection between the invariant manifolds (the intersection of these manifolds at $a \approx 0.37427117$ in the three-dimensional state space of the boundary map manifests itself as a tangency between the corresponding projections), compare with Figure~\ref{fig:set-valued_eps06}(b). At $a = 0.38$, after the topological bifurcation has taken place and the attractor has exploded, the projection of the unstable manifold (blue) of the saddle fixed point (red) appears to be unbounded in (b).}
    \label{fig:a0.37_0.38}
\end{figure}

However, whereas the topological bifurcation in Section~\ref{sec:top_bif} corresponds to a fold bifurcation of the boundary map, here the boundary map exhibits 
a heteroclinic bifurcation, where a heteroclinic cycle is born.

We illustrate the situation in detail in Figure~\ref{fig:a0.37_0.38} for (a) $a = 0.37$ and (b) $a = 0.38$. We depict the relevant fixed points and parts of their invariant manifolds of the boundary map: a fixed point of saddle type (in red) with a one-dimensional unstable manifold (in blue) and an attracting fixed point (in green, in the interior of the attractor, partially hidden between projection of the unstable manifold). 
The boundary of its domain of attraction (and boundary of its dual repeller) is related to the one-dimensional stable manifold (in orange) of a fixed point of saddle type (in purple) with a two-dimensional unstable manifold.  The correspondence between Figure~\ref{fig:a0.37_0.38}(a) and the boundaries of the minimal attractor and its domain of attraction in Figure~\ref{fig:set-valued_eps06}(b) is readily verified.

\begin{figure}[!b]
    \begin{overpic}[width=.49\textwidth]{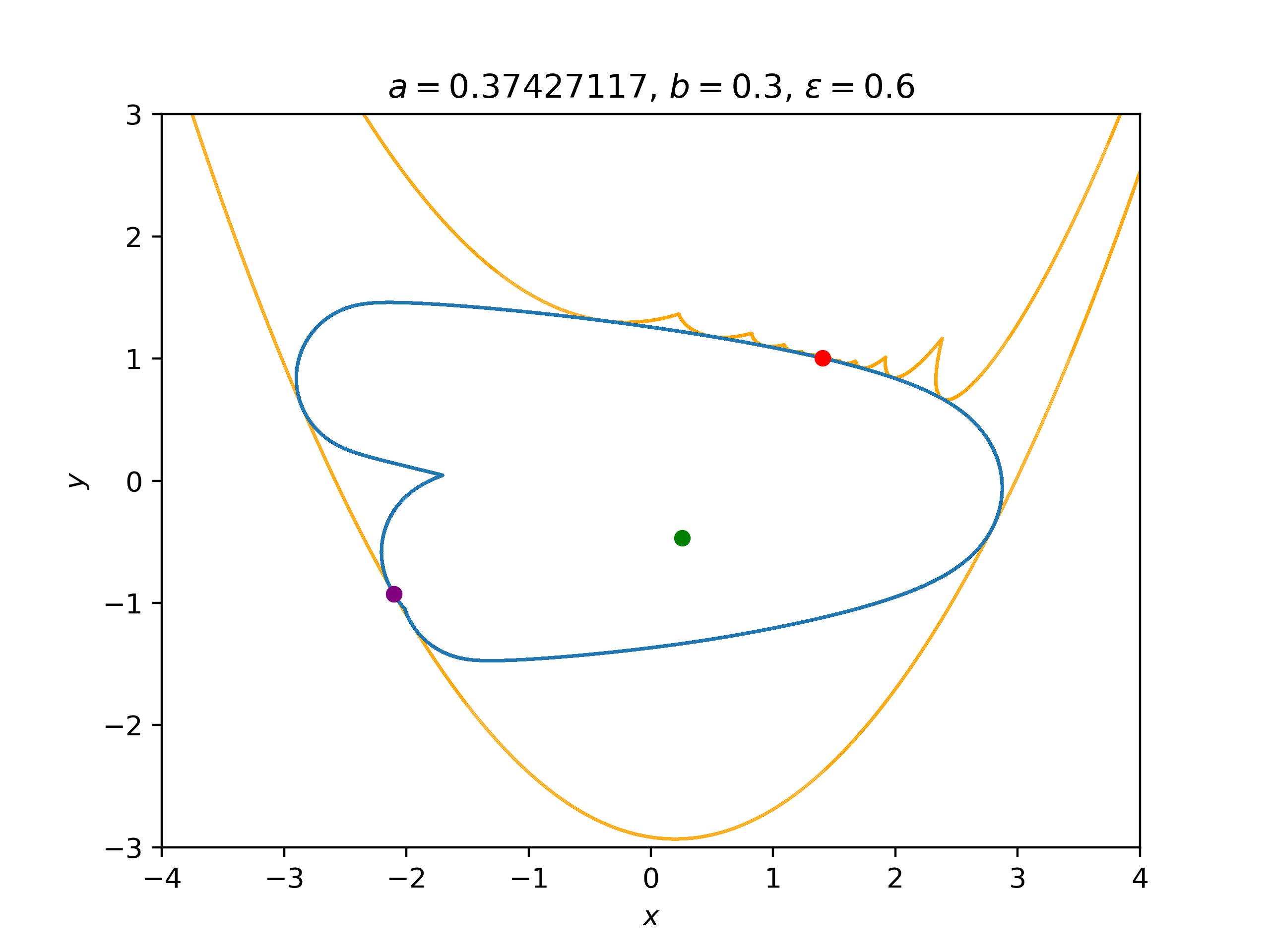}
    
    \put(0,0){
    (a)
    }
    \end{overpic}
    \begin{overpic}[width=.49\textwidth]{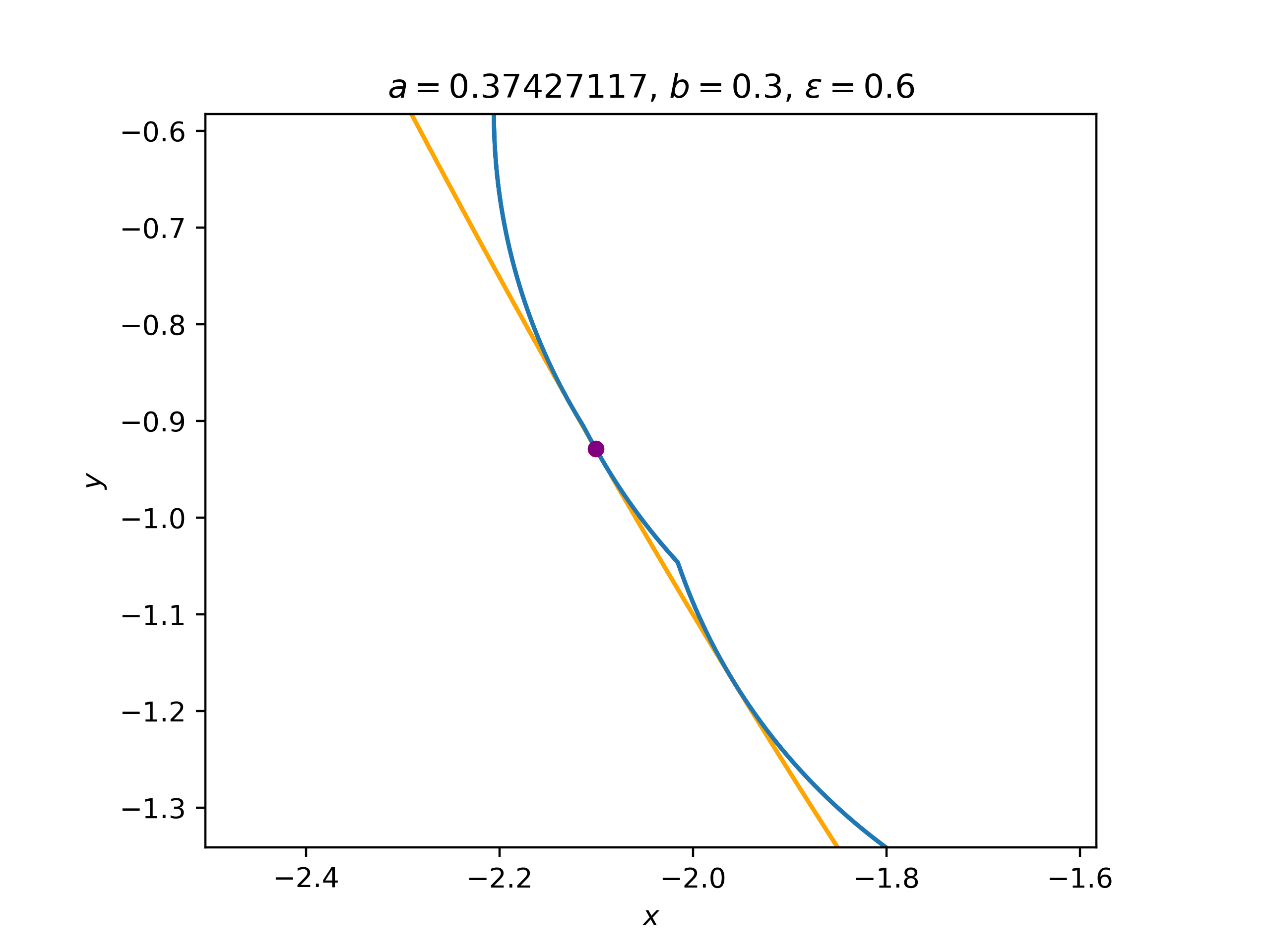}
    
    \put(0,0){
    (b)
    }
    \end{overpic}

    \captionsetup{width=\linewidth}\caption{The boundary of the minimal attractor and its dual repeller represented by (parts of) relevant projections of invariant manifolds at the parameter values: $b=0.3, \varepsilon = 0.6, a = 0.37427117$, on the verge of a topological bifurcation. 
    The projections to the interior of the minimal attractor or its dual repeller are omitted for clarity. A cascade of wedge singularities is found on both sides of the boundary of the dual repeller near the saddle fixed point (red dot in (a)). There is also a cascade of wedge singularities on both sides of the boundary of the minimal attractor near the other saddle fixed point (purple dot in (a) and the magnification (b), where tangencies are visible only on one side of the saddle point due to the small scale at which they arise on the other side). In  \cite{lamb2020boundaries}, boundary points with cascades of wedges on both sides are called \textit{shallow-shallow} singularities.}
    \label{fig:a0.37427117_eps0.6}
\end{figure}

In Figure~\ref{fig:a0.37_0.38}(a), the projection of the unstable manifold (in blue) of the red fixed point appears to be on the verge of intersecting with the projection of the stable manifold (in orange) of the purple fixed point. Indeed, when $a$ is increased further, at $a \approx 0.37427117$ 
the unstable manifold forms a heteroclinic cycle with the stable manifold  (also cf. Figure~\ref{fig:a0.37427117_eps0.6}). The birth of this heteroclinic cycle of the boundary map corresponds to a topological bifurcation of the minimal attractor, where the minimal attractor disappears after the bifurcation. 
The heteroclinic intersections represent collision points between the minimal invariant set and its dual repeller (when the minimal invariant set loses its attractivity). This is a known necessary condition for topological bifurcation~\cite{lamb2015topological}.

We present a continuation of the saddle fixed points and their respective one-dimensional unstable and stable manifolds for $a = 0.38$ in Figure~\ref{fig:a0.37_0.38}(b): 
the heteroclinic cycle between the unstable manifold of the red fixed point and the stable manifold of the purple fixed point is broken, leading to the topological bifurcation. The unstable manifold of the red fixed point appears to be unbounded (towards the left edge of Figure~\ref{fig:a0.37_0.38}(b)). This signifies that the minimal attractor from $a = 0.37$ (in Figure~\ref{fig:a0.37_0.38}(a)) explodes, and has disappeared after the bifurcation.

We now focus on the parameter value $a=0.37427117$, which is very close to (just before) the topological bifurcation. 
In Figure~\ref{fig:a0.37427117_eps0.6} we show the boundary of the minimal attractor and its dual repeller as (parts of) the projection of invariant manifolds of the boundary map, omitting parts of the invariant manifolds that are irrelevant.

We observe a two-sided cascade of wedge singularities on the boundary of the dual repeller (in orange), near the projection of the saddle fixed point (in red).
There is also a two-sided cascade of wedge singularities on the boundary of the minimal attractor 
 (in blue), near the projection of the saddle fixed point (in purple). On the magnification in Figure~\ref{fig:a0.37427117_eps0.6}(b), the heteroclinic tangencies are only visible on one side of the saddle point as they arise on a very small scale on the other side.
This points to the existence of a \emph{shallow-shallow singularity}~\cite{lamb2020boundaries} on both the boundary of the dual repeller and the minimal invariant set, observed only at the point of topological bifurcation (where the invariant set loses its attractivity). 
The genericity of this kind of topological bifurcation, in relation to the underlying geometry, requires further theoretical explanation, which is beyond the scope of this paper.

\subsection{Creation of a shallow singularity (and a cascade of wedge singularities) through non-transversal intersections of the unstable manifold of a saddle fixed point of the boundary map, with the strong stable foliation of a stable periodic point}\label{sec:cascade}

Finally, we explore an instance of a boundary bifurcation with higher complexity. 

\begin{figure}[!htbp]
\centering
    \begin{overpic}[width=.49\textwidth]{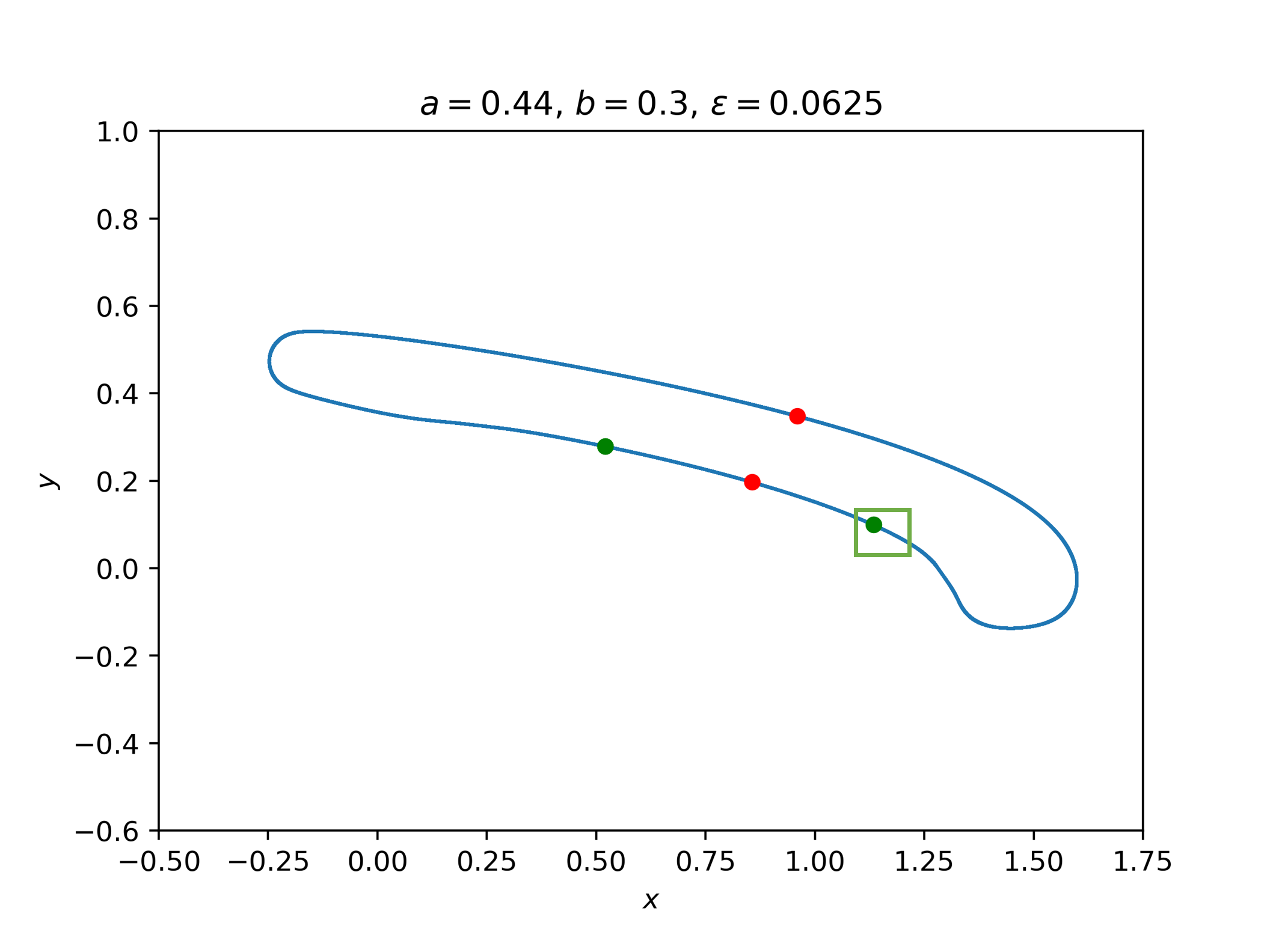}
    
    \put(0,0){
    (a)
    }
    \end{overpic}
    \begin{overpic}[width=.49\textwidth]{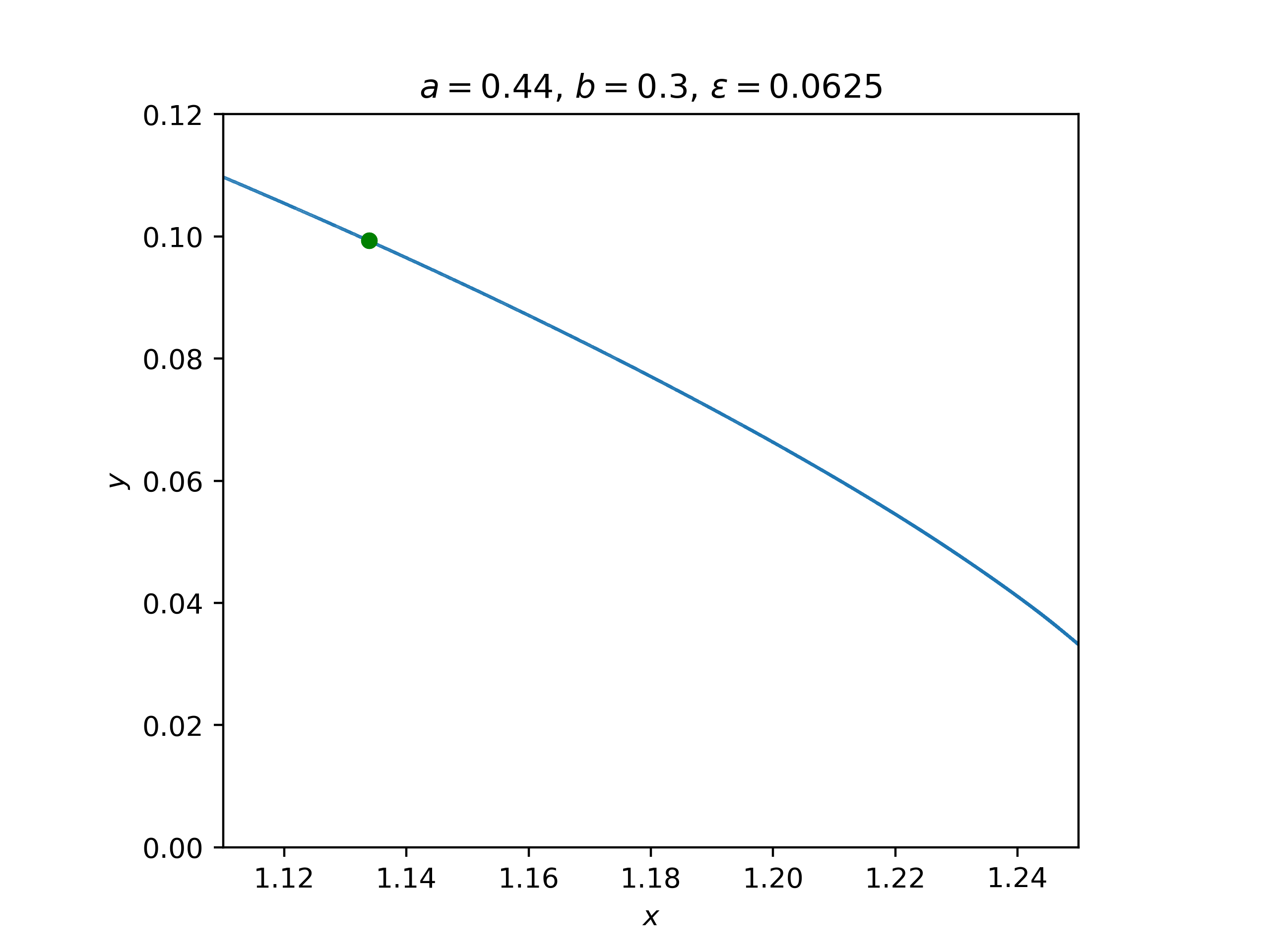}
    
    \put(0,0){
    (b)
    }
    \end{overpic}
\begin{minipage}{.007\textwidth}
    \text{}
\end{minipage}

\vspace{.3em}
    \begin{overpic}[width=.49\textwidth]{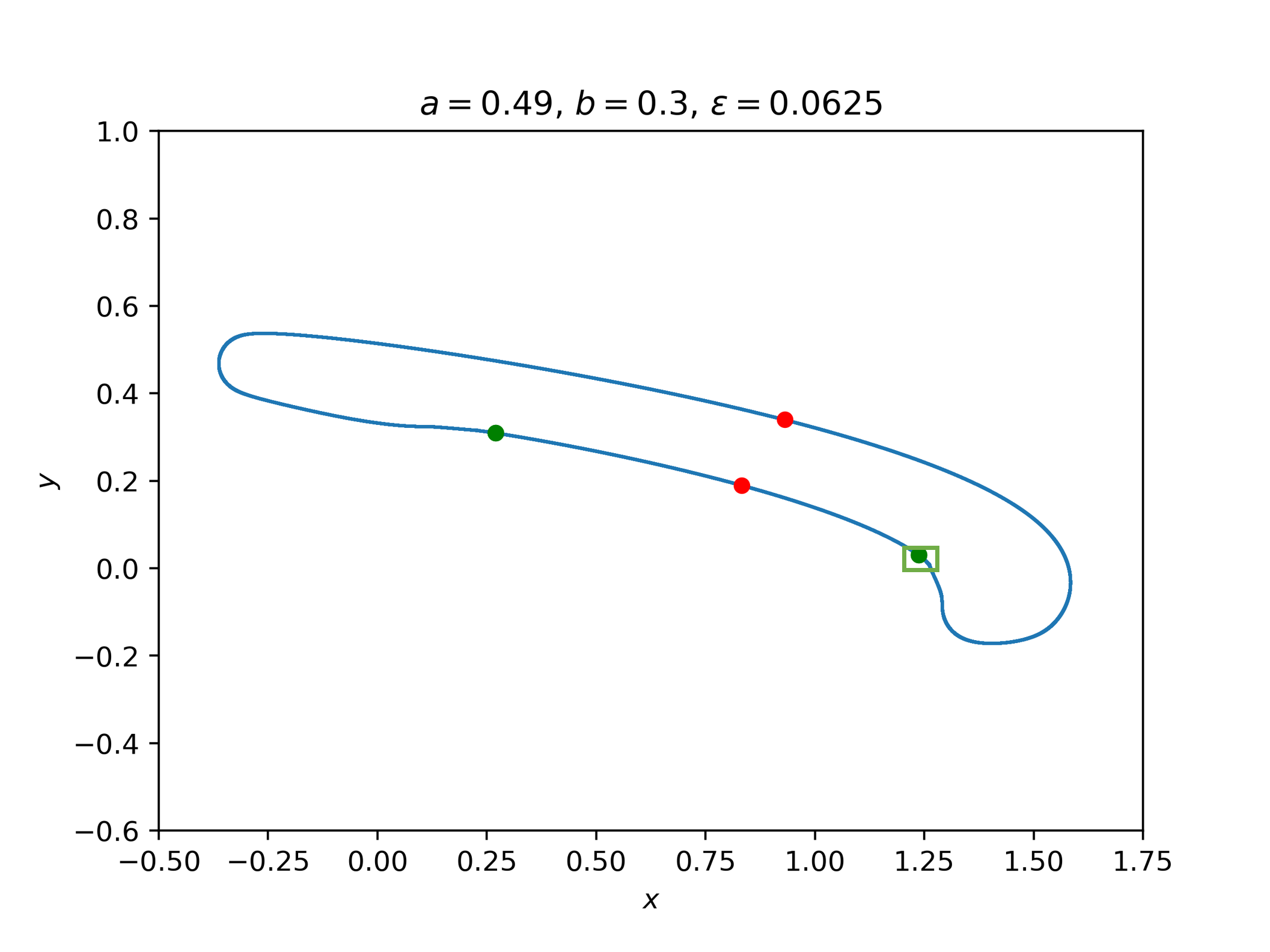}
    
    \put(0,0){
    (c)
    }
    \end{overpic}
    \begin{overpic}[width=.49\textwidth]{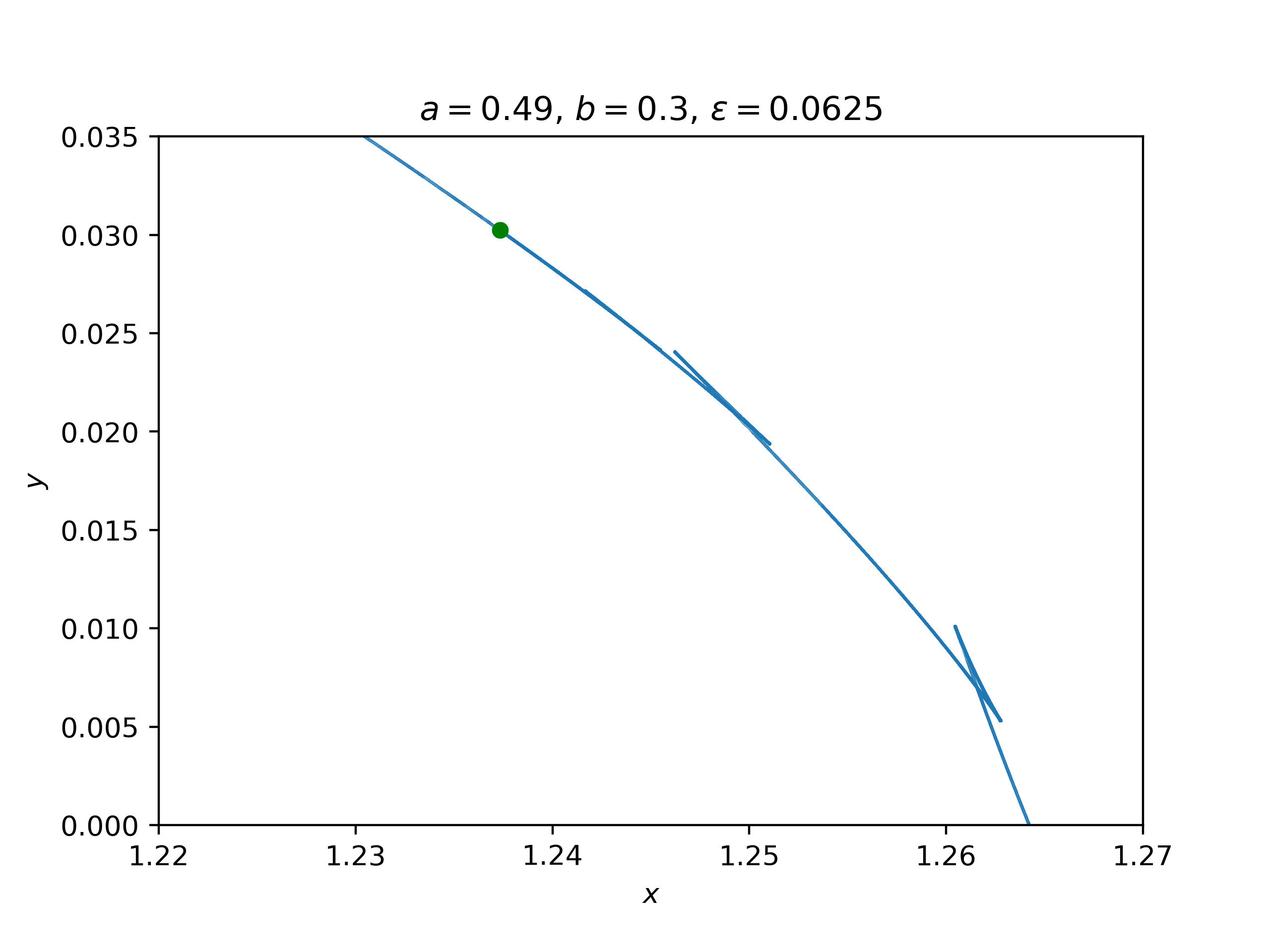}
    
    \put(0,0){
    (d)
    }
    \end{overpic}
    
    \captionsetup{width=\linewidth}\caption{The boundary of the minimal attractor at parameter values $b=0.3$, $\varepsilon=0.0625$ and (a) $a=0.44$, (c) $a=0.49$, represented as (parts of) projections of a heteroclinic cycle (blue curve) between two saddle fixed points (red dots) and a stable two-periodic orbit (green dots). The magnifications of the small boxed regions near one of the stable periodic points are shown in (b) and (d), respectively. 
    At (b) $a = 0.44$, one observes a smooth boundary, whereas at (d) $a = 0.49$ the projection of the heteroclinic cycle self-intersects, leading to wedge singularities on the boundary of the minimal attractor.}
    \label{fig:henon044_049}
\end{figure}

Consider the parameter values $a = 0.44$, $b=0.3$ and $\varepsilon = 0.0625$. We find that the boundary of the minimal attractor is smooth and its normal bundle corresponds to a heteroclinic cycle between two saddle fixed points and a stable two-periodic orbit of the boundary map. Figure~\ref{fig:henon044_049}(a) illustrates the boundary of the minimal attractor as the projection of the unstable manifolds of two saddle fixed points (in red) connecting to a stable two-periodic orbit (in green). The projection of the heteroclinic cycle shows no self-intersections, indicative of smoothness of the boundary of the minimal attractor, see also the magnification in Figure~\ref{fig:henon044_049}(b). 

When the parameter value of $a$ is increased from $a=0.44$ to $a = 0.49$ a boundary bifurcation is observed. The relevant continuation of the heteroclinic cycle in Figure~\ref{fig:henon044_049}(a) is depicted in Figure~\ref{fig:henon044_049}(c). 
Examination near one of the stable periodic points in Figure~\ref{fig:henon044_049}(d), reveals self-intersections on the projection of the unstable manifold, corresponding to singularities on the boundary of the minimal attractor. 

On closer inspection, when zooming in, one observes what seems to be an infinite sequence of increasingly shallow wedge singularities cascading towards the periodic point.

\begin{figure}[!tbp]
\centering
    
    \includegraphics[width=.5\textwidth,trim={0 20 0 20},clip]{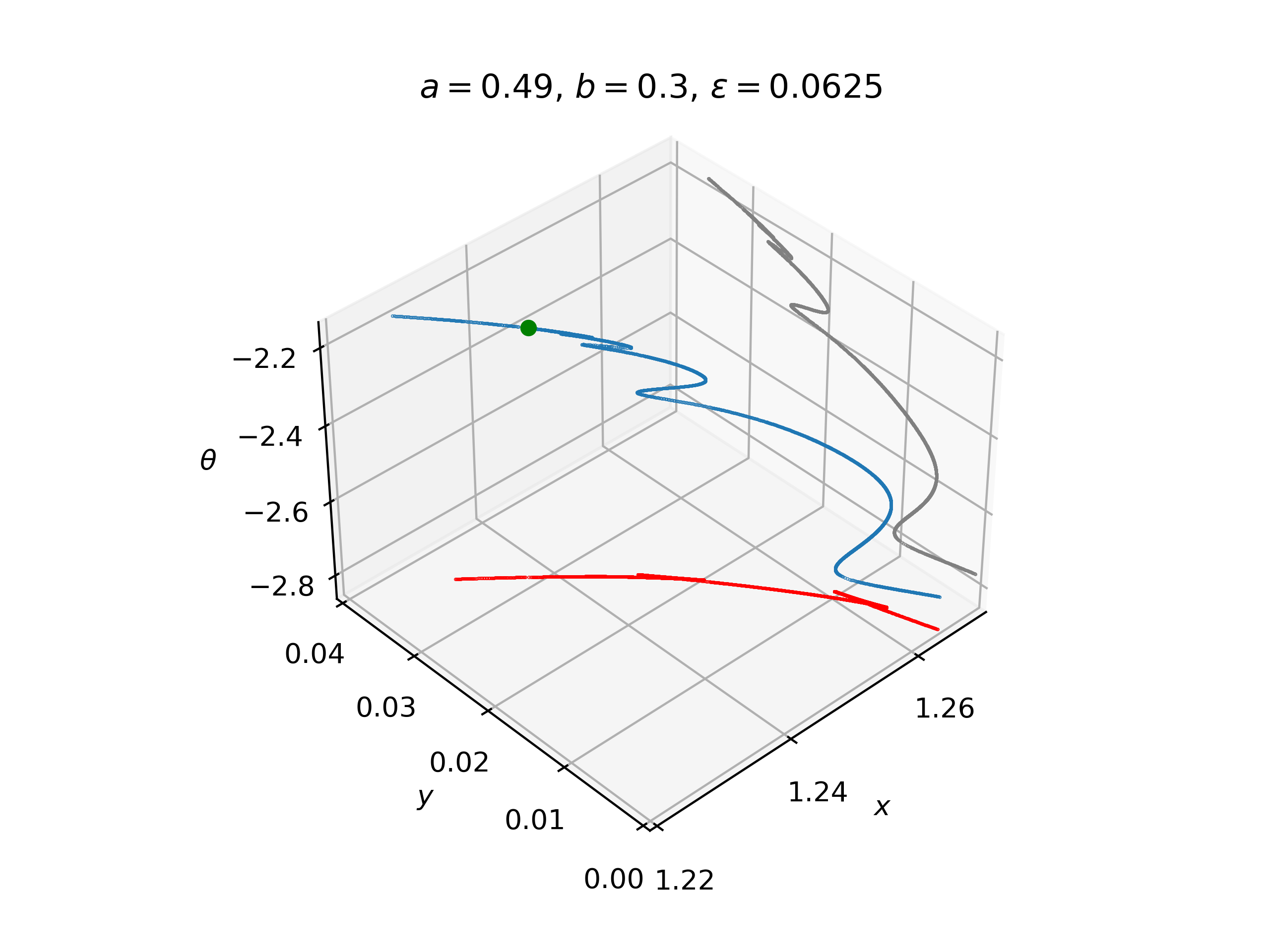}
    
    \captionsetup{width=\linewidth}\caption{The invariant manifold of the boundary map (blue) near to the two-periodic stable point (green), with relevant projection (red)  corresponding to Figure~\ref{fig:henon044_049}(d). A projection to an additional coordinate plane (grey) is included for perspective.}
    \label{fig:henon_049_3d}
\end{figure}

\begin{figure}[!tbp]
    \centering
    \includegraphics[width=\linewidth]{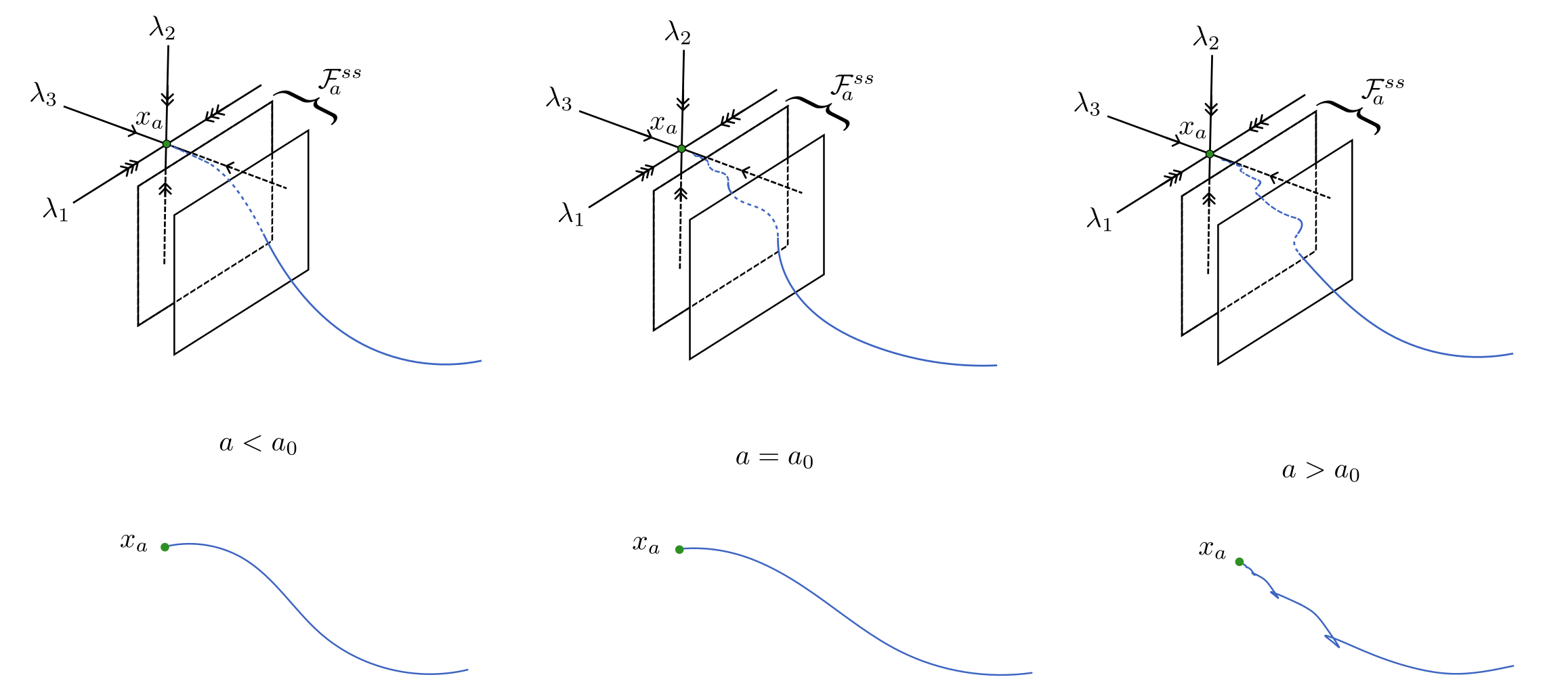}
    \captionsetup{width=\linewidth}\caption{Sketch of the transition from transverse to non-transverse intersection of an incoming one-dimensional unstable manifold of one fixed point of the boundary map (not depicted) and the two-dimensional strong stable foliation $\mathcal{F}^{ss}_a$ of another stable fixed point $x_a$, with $a_0$ denoting the bifurcation point. 
    The bottom figures represent the corresponding relevant projections (as observed in numerical experiments), leading to the conjecture of a cascade of wedge singularities accumulating to $x_a$ if $a>a_0$.}
    \label{fig:cascade_wedge}
\end{figure}

The existence and persistence of such a cascade may be understood from the point of view of the boundary map. In Figure~\ref{fig:henon_049_3d}, we present a visualisation of the unstable manifold (blue) of a fixed point of the boundary map in $\mathbb{R}^2 \times S^1$, the projection of which (red) yields Figure~\ref{fig:henon044_049}(d). 
As the unstable manifold approaches the stable periodic point, it `wiggles'. This wiggling appears to persist arbitrarily close to the periodic point.
Such infinite wiggling indeed arises persistently if the unstable manifold of the fixed point has a non-transversal intersection with the strong stable foliation of the stable two-periodic orbit, 
as illustrated schematically in Figure~\ref{fig:cascade_wedge}. 
It appears that due to the underlying geometry of the boundary map, the relevant projection of each such a wiggle leads to a self-intersection and thus to a wedge singularity on the boundary, leading to an infinite cascade of wedge singularities accumulating to the projection of the periodic point. A formal proof requires a deep appreciation of the underlying geometry, which is beyond the scope of this paper. 
The possibility of such a cascade was already noted in \cite{lamb2020boundaries}, where the accumulation point was coined a \emph{shallow singularity}. Interestingly, the dynamical setting indicates that such accumulation points of singularities can arise persistently, whereas this would not be expected from singularity theoretical considerations of projections of smooth curves. 

Please note that the shallow singularity found here features a one-sided cascade, in contrast to the double-sided cascade at the shallow-shallow singularity encountered in Section~\ref{sec:hetero}.
The latter was not found to be persistent, but only arise at the bifurcation point.

As this boundary bifurcation requires an appreciation of the global positioning of invariant manifolds, it is not easy to identify the precise bifurcation point. This contrasts with the boundary bifurcation discussed in Section~\ref{sec:complex}, where the bifurcation point coincides with a change of the type of eigenvalues of the Jacobian of the boundary map at a relevant fixed point.

\begin{figure}[!b]
    
    \centering\includegraphics[width=.6\textwidth]{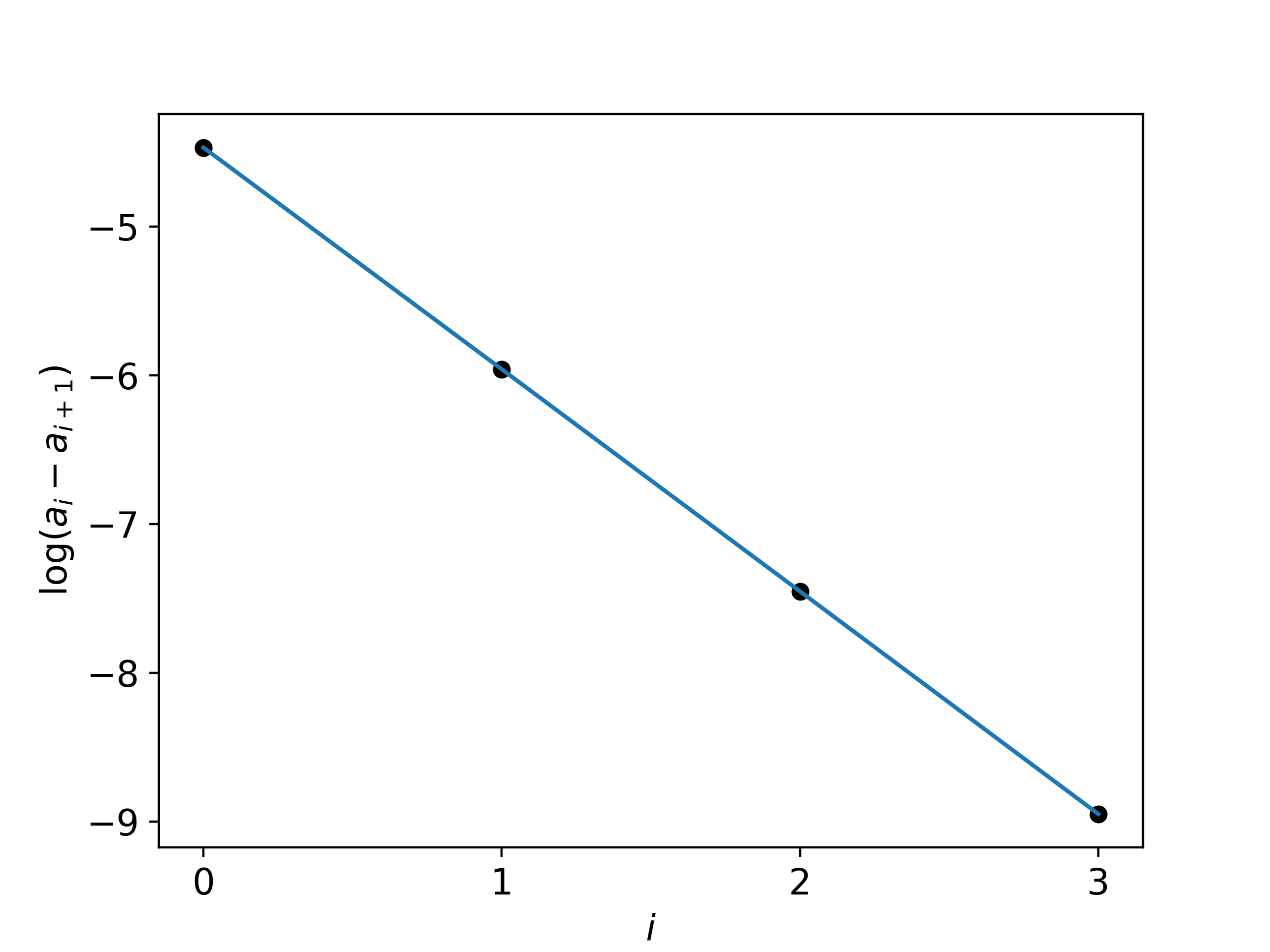}
    
    \captionsetup{width=\linewidth}\caption{Log-distance between consecutive parameter values where the outermost wedge singularity of a cascade of wedge singularities, accumulating to a shallow singularity, disappears. Close to the boundary bifurcation where the shallow singularity vanishes, the graph evidences an exponential decrease, enabling an estimate of the bifurcation parameter in \eqref{eq:bifpoint}.}
    
    \label{fig:param_singularity}
\end{figure}

However, we observe that the disappearance of the shallow singularity is characterised by a scaling relation when the parameter $a$ is gradually decreased towards the bifurcation point from 0.49 to around 0.44. 
We observe the cascade disappears by the mechanics of successive boundary bifurcations where, one by one, the outermost wedge singularity of the cascade, furthest from the shallow singularity, disappears.

From the geometry in the sketch in Figure~\ref{fig:cascade_wedge}, we anticipate an asymptotic scaling of the successive parameter values at which this happens. Indeed, measuring the first five parameter values at which the first wedge of the cascade disappears, yields - as shown in Figure~\ref{fig:param_singularity} - an exponential relationship.
From this relationship, 
\begin{displaymath}
    \log(a_i-a_{i+1})=\log(a_0-a_1)+ci,
\end{displaymath}
for $i \in \mathbb{N}$ with $c\approx -1.48933,\; a_0\approx 0.46964$ and $a_1 \approx 0.45820$, we find that the boundary bifurcation where the shallow singularity is created, occurs at parameter value
\begin{equation}\label{eq:bifpoint}
a_{\infty} = a_1 - \sum_{i=1}^{\infty}e^{\log(a_i-a_{i+1})}
= a_1 - \frac{a_0-a_1}{e^{-c}-1} \approx 0.454869.    
\end{equation}

\section{Outlook}\label{sec:outlook}
In this paper, we have employed the boundary map \eqref{eq:boundary_map} (first introduced in \cite{kourliouros2023persistence}) as a tool to numerically study attractors of random dynamical systems with bounded noise and their bifurcations. The H\'enon map with bounded noise \eqref{eq:randomdiff} was chosen as a prototypical example, for illustrative purposes.

The numerics have laid bare various interesting observations that warrant further theoretical and numerical work.

From the theoretical point of view, the aim is to develop a bifurcation theory, akin to the single-valued setting, that identifies generic features of attractors, and classifies their bifurcations (especially those of low co-dimension, that persist in one- and two-parameter families of systems). The boundary map provides a roadmap towards this objective, through connections between dynamical features of the boundary map and corresponding topological and boundary bifurcations of attractors in the random (set-valued) context. 

As a first step,  \cite{kourliouros2023persistence} addresses the smooth persistence of minimal invariant sets with smooth boundaries, using normal hyperbolicity and relevant insights from contact geometry (that underlies the setting). Extensions of these results are expected to yield a classification of attractors that are persistent in terms of topology and boundary singularity structure. 
Indeed, the present paper contains various conjectures in this direction, supported by numerics and intuitive arguments. 

Another separate challenge lies in the numerical approximation and continuation of invariant sets and attractors. In our experience, the boundary map is more efficient and more accurate than brute-force set-valued methods (like by employing GAIO), especially when studying bifurcations.
The boundary map provides an opportunity to build on existing numerical tools for bifurcations of finite dimensional deterministic dynamical systems, in the absence of any feasible approach from the set-valued point of view.

Finally, we would like to remark that while we presented the boundary map in the context of dynamical systems with bounded noise, a similar 
 set-valued dynamical systems point of view naturally arises in control theory,  uncertainty quantification and front propagation, where the boundary map may also be of use.

\section*{Acknowledgments}
The authors are grateful to Konstantinos Kourliouros, Kalle Timperi and Dmitry Turaev for useful discussions, and Kalle Timperi in particular for his assistance with Figures~\ref{fig:set-valued_illustrate} and \ref{fig:contribute}. JSWL and MR have been supported by the EPSRC grants EP/W009455/1 and  EP/Y020669/1.  JSWL also acknowledges support from the EPSRC Centre for Doctoral Training in Mathematics of Random Systems: Analysis, Modelling and Simulation (EP/S023925/1) and thanks IRCN (Tokyo) and GUST (Kuwait) for their support. JSWL and WHT are grateful for support from JST Moonshot R \& D Grant Number JPMJMS2021, and  WHT was supported also by EPSRC PhD scholarship grant EP/S515085/1 and the Project of Intelligent Mobility Society Design, Social Cooperation Program, UTokyo.

\bibliography{reference}

\newpage

\appendix
\include{appendix_final}

\end{document}

%% file: appendix_final.tex
\renewcommand{\thesection}{\Alph{section}}

\section{Proof of Proposition~\ref{prop:dual_invariant}}\label{appb}

We consider random dynamical systems consisting of a diffeomorphism $f$ with additive noise of reach $\varepsilon>0$, as in \eqref{eq:randomdiff}, and denote the corresponding set-valued map \eqref{eq:setvaluedmap} and associated boundary map (Definition~\ref{def:boundary_mapping}), as $F_f$ and $\beta_f$, respectively, to contrast these with the corresponding maps for the inverse $f^{-1}$, i.e.~$F_{f^{-1}}$ and $\beta_{f^{-1}}$. 

The following Lemma aids the proof of Proposition~\ref{prop:dual_invariant}.
\begin{lemma}\label{lemma:image_invariance}
    Let $A$ be an $F_f$-invariant set with continuously differentiable boundary $\partial A$. Then, the inner unit normal bundle $N_1^-\partial f(A)$ of $\partial f(A)$ is $\beta_{f^{-1}}$-invariant.
\end{lemma}

\begin{proof}
    From Proposition~\ref{prop:invariant}, the outward unit normal bundle is $\beta_f$-invariant, i.e. 
    
    \noindent $\beta_f (N_1^+\partial A) = N_1^+\partial A$. The boundary map $\beta_f$ can be expressed as the composition $\beta_f = g_{\varepsilon}\circ h_{f}$ where
    \begin{equation*}
        h_f(x,n) = \left(f(x),\frac{(f'(x)^T)^{-1}n}{||(f'(x)^T)^{-1}n||}\right), \quad \mbox{and}\quad g_{\varepsilon}(x,n) = (x+\varepsilon n, n).
    \end{equation*}
    Then, $h_f(N_1^+\partial A) = N_1^+f(\partial A)$ and $N_1^+\partial A = g_{\varepsilon}(N_1^+f(\partial A))$. Hence, 
    \begin{align}\label{eq:proof1}
        N_1^+f(\partial A) &= h_f(g_{\varepsilon}(N_1^+f(\partial A))) 
    \end{align}
    and, equivalently,
    \begin{align}\label{eq:proof2}
         N_1^+f(\partial A) &= g^{-1}_{\varepsilon}(h_f^{-1}(N_1^+f(\partial A))). 
    \end{align}

From (\ref{eq:proof2}) and $h_{f^{-1}}(x,n) =h_f^{-1}(x,n)$, we find for all $(x,-n) \in N_1^+f(\partial A)$ (i.e.~$(x,n) \in N_1^-f(\partial A)$),
    \begin{align*}
        g^{-1}_{\varepsilon}(h_f^{-1}(x,-n)) &=g^{-1}_{\varepsilon}\left(f^{-1}(x),\frac{((f^{-1})'(x)^T)^{-1}(-n)}{\|((f^{-1})'(x)^T)^{-1}(-n)\|}\right)\\
        &=\left(f^{-1}(x)+\varepsilon \frac{((f^{-1})'(x)^T)^{-1}n}{\|((f^{-1})'(x)^T)^{-1}n\|},-\frac{((f^{-1})'(x)^T)^{-1}n}{\|((f^{-1})'(x)^T)^{-1}n\|}\right) \in N_1^+f(\partial A),
    \end{align*}
    so that 
    \begin{align*}
        g_{\varepsilon}(h_{f^{-1}}(x,n)) &= \left(f^{-1}(x)+\varepsilon \frac{((f^{-1})'(x)^T)^{-1}n}{\|((f^{-1})'(x)^T)^{-1}n\|},\frac{((f^{-1})'(x)^T)^{-1}n}{\|((f^{-1})'(x)^T)^{-1}n\|}\right) \in N_1^-f(\partial A).
    \end{align*}
    Hence, $\beta_{f^{-1}}(N_1^-f(\partial A)) = g_{\varepsilon}\left(h_{f^{-1}}\left(N_1^-f(\partial A)\right)\right) \subset N_1^-f(\partial A)$. 
    
    Using similar arguments and (\ref{eq:proof1}), we obtain for all $(x,n) \in N_1^-f(\partial A)$
    \begin{align*}
        h_f(g_{\varepsilon}(x,-n)) &= \left(f(x-\varepsilon n),-\frac{(f'(x-\varepsilon n)^T)^{-1}n}{\|(f'(x-\varepsilon n)^T)^{-1}n\|}\right) \in N_1^+f(\partial A),
        \end{align*}
       so that 
        \begin{align*}       
        h_{f^{-1}}^{-1}(g_{\varepsilon}^{-1}(x,n)) &= \left(f(x-\varepsilon n),\frac{(f'(x-\varepsilon n)^T)^{-1}n}{\|(f'(x-\varepsilon n)^T)^{-1}n\|}\right) \in N_1^-f(\partial A).
    \end{align*}
    Therefore, $\beta^{-1}_{f^{-1}}(N_1^-f(\partial A)) \subset N_1^-f(\partial A)$ and hence we have $\beta_{f^{-1}}(N_1^-f(\partial A)) = N_1^-f(\partial A)$ as required.
\end{proof}

 The proof of Proposition~\ref{prop:dual_invariant} follows from the observation that for any invariant set $A^*$ of the dual $F_f^*$, the set $\overline{B_{\varepsilon}(A^*)}$ is invariant under $F_{f^{-1}}$: $F_{f^{-1}}(\overline{B_{\varepsilon}(A^*)}) = F_f^*(A^*) = A^*$. Namely, by Lemma \ref{lemma:image_invariance}, this implies that $N_1^-\partial f^{-1}(\overline{B_{\varepsilon}(A^*)}) = N_1^-\partial A^*$ is invariant under $\beta_{(f^{-1})^{-1}} = \beta_f$.

\section{Properties of eigenvalues and eigenvectors of the derivative at fixed points of the boundary map}\label{sec:eigen}

\begin{proposition}\label{prop:eigenrelation}
Let $\beta$ be the boundary map \eqref{eq:boundary_map} associated with \eqref{eq:randomdiff}, $(x_*,n_*)$ be a fixed point of $\beta$ with derivative $D\beta(x_*,n_*):T_{x_*}\mathbb{R}^2\times T_{n_*}S^1\to T_{x_*}\mathbb{R}^2\times T_{n_*}S^1$, and $\pi:T_{x_*}\mathbb{R}^2\times T_{n_*}S^1\to T_{x_*}\mathbb{R}^2$ denote the canonical projection.
Then $\lambda_1 := \|(f'(x_*)^T)^{-1}n_*\|^{-1}\neq0$ is an eigenvalue of $D\beta(x_*,n_*)$.

 In addition, assume that  $\lambda_1$ is the only real eigenvalue 
 of $D\beta(x_*,n_*)$, 

and let $V$  denote the two-dimensional $D\beta(x_*,n_*)$-invariant subspace transverse to the eigenspace associated with $\lambda_1$. Then, $\dim(\pi(V))=1$ and $\lambda_1 = \det(D\beta(x_*,n_*)|_V)$. 
\end{proposition}

\begin{proof}
The demonstration consists of  a direct computation of the eigenvalues and eigenvectors at fixed points of the boundary map. For convenience we consider the embedding $\mathbb{R}^2\times S^1:=\{(x,n)\in\mathbb{R}^2\times\mathbb{R}^2~|~\|n\|=1\}\subset \mathbb{R}^4$. We denote the corresponding representation of  the 
boundary map in $\mathbb{R}^4$ in this embedded setting 
$$\hat{\beta}(x,n) := (\hat{\beta}_1(x,n),\hat{\beta}_2(x,n)),$$ where 
\begin{equation*}
    \hat{\beta}_1(x,n) := f(x)+\varepsilon \hat{\beta}_2(x,n), \quad \hat{\beta}_2(x,n) := 
{\|\left(f'(x)^T\right)^{-1}n\|}^{-1}
\left(f'(x)^T\right)^{-1}n.
\end{equation*}
By direct calculation, we find that
\begin{align*}
    D\hat{\beta}(x,n)=\begin{pmatrix}
        f'(x) + \varepsilon D_x(\hat{\beta}_2(x,n)) & \varepsilon D_n(\hat{\beta}_2(x,n))\\
        D_x(\hat{\beta}_2(x,n)) & D_n(\hat{\beta}_2(x,n)),
    \end{pmatrix}
\end{align*}
where $f'$ denotes the derivative of $f$, and
\begin{align*}
    D_x(\hat{\beta}_2(x,n)) &= \frac{I-\hat{\beta}_2(x,n)\hat{\beta}_2(x,n)^T}{\|\left(f'(x)^T\right)^{-1} n\|}D_x\left(\left(f'(x)^T\right)^{-1}n\right), \\
    D_n(\hat{\beta}_2(x,n)) &= \frac{I-\hat{\beta}_2(x,n)\hat{\beta}_2(x,n)^T}{\|\left(f'(x)^T\right)^{-1} n\|}\left(f'(x)^T\right)^{-1}. 
\end{align*}
where $I$ denotes the identity map. 

We choose $u_n \in \mathbb{R}^2$ to be a unit vector orthogonal to $n$, so that $u_n^Tn = 0$, and we choose a unit vector $u_{\hat{\beta}_2}$ so that $u_{\hat{\beta}_2}^T\hat{\beta}_2 = 0$, suppressing the dependence of $\hat{\beta}_2$ on $(x,n)$. Then, using the fact that
    $\hat{\beta}_2\hat{\beta}_2^T + u_{\hat{\beta}_2}u^T_{\hat{\beta}_2} = I$,
we obtain
\begin{align*}
    D\hat{\beta}(x,n) &= \begin{pmatrix}
        f'(x) + \varepsilon \frac{u_{\hat{\beta}_2}u^T_{\hat{\beta}_2}D_x\left(\left(f'(x)^T\right)^{-1}n\right)}{\|\left(f'(x)^T\right)^{-1}n\|} & \varepsilon \frac{u_{\hat{\beta}_2}u^T_{\hat{\beta}_2}\left(f'(x)^T\right)^{-1}}{\|\left(f'(x)^T\right)^{-1}n\|}\\
        \frac{u_{\hat{\beta}_2}u^T_{\hat{\beta}_2}D_x\left(\left(f'(x)^T\right)^{-1}n\right)}{\|\left(f'(x)^T\right)^{-1}n\|}& \frac{u_{\hat{\beta}_2}u^T_{\hat{\beta}_2}\left(f'(x)^T\right)^{-1}}{\|\left(f'(x)^T\right)^{-1}n\|}
    \end{pmatrix}\\
    &= \begin{pmatrix}f'(x)&0_{2\times 2}\\0_{2\times 2}&0_{2\times 2}
    \end{pmatrix} + \frac{1}{\|\left(f'(x)^T\right)^{-1}n\|} \times
    \\&~~~~~\begin{pmatrix}\varepsilon u_{\hat{\beta}_2}\\u_{\hat{\beta}_2}\end{pmatrix} \begin{pmatrix}u_{\hat{\beta}_2}^T D_x\left(\left(f'(x)^T\right)^{-1}n\right)&u_{\hat{\beta}_2}^T\left(f'(x)^T\right)^{-1}\end{pmatrix}, 
\end{align*}
with $0_{n\times m}$ denoting the $n\times m$ zero matrix.

It is readily verified that the relationship between $D\beta(x,n)$ and its embedded version $D\hat{\beta}(x,n)$ is given by
\begin{equation}\label{eq:Jacobian_general}
    D\beta(x,n) = \begin{pmatrix}I&0_{2\times 2}\\0_{1\times 2}&u_{\hat{\beta}_2}^T\end{pmatrix}D\hat{\beta}(x,n)\begin{pmatrix}I&0_{2\times 1}\\0_{2\times 2}&u_n\end{pmatrix},
\end{equation}
where $I$ denotes the two-dimensional identity.

Now, let  $(x_*,n_*)$ be a fixed point of $\beta$, then 
\begin{equation*}
    n_* = \beta_2(x_*,n_*) = \frac{(f'(x_*)^T)^{-1}n_*}{\|(f'(x_*)^T)^{-1}n_*\|} \iff (f'(x_*)^T)^{-1}n_* = \|(f'(x_*)^T)^{-1}n_*\|n_*.
\end{equation*}
Hence, $\lambda_1 :=\|(f'(x_*)^T)^{-1}n_*\|^{-1}$ is an eigenvalue of  $f'(x_*)$ (and $f'(x_*)^T$).

We assert that $u_{n_*}$ is an eigenvector of $f(x_*)$. Indeed, from $$\left\langle\lambda_1^{-1}n_*,f'(x_*)u_{n_*}\right\rangle = \left\langle\left(f'(x_*)^T\right)^{-1}n_*,f'(x_*)u_{n_*}\right\rangle = \langle{n_*},u_{n_*}\rangle = 0, $$ the vector $f'(x_*)u_{n_*}$ is orthogonal to $n_*$
and hence
parallel to $u_{n_*}$. 
Let $\lambda$ denote the eigenvalue of $f'(x_*)$ associated to $u_{n_*}$.
Then, $$u^T_{n_*}\left(f'(x_*)^T\right)^{-1}u_{n_*} = \left(f'(x_*)^{-1}u_{n_*}\right)^Tu_{u_*} = \lambda_1^{-1}u_{n_*}^Tu_{n_*} = \lambda^{-1}.$$

To aid notation, let $w^T := u^T_{n_*}D_{x_*}\left(\left(f'(x_*)^T\right)^{-1}n_*\right),$ then the derivative (\ref{eq:Jacobian_general}) at the fixed point $(x_*,n_*)$ can be written as
\begin{align}\label{eq:Jacobian1}
    D\beta(x_*,n_*)
    &= \begin{pmatrix}
        f'(x_*)&0_{2\times 1}\\
        0_{1\times 2} & 0_{1 \times 1}
    \end{pmatrix} + \lambda_1\begin{pmatrix}
        \varepsilon u_{n_*}\\
        1
    \end{pmatrix}\begin{pmatrix}
        w^T&\lambda^{-1}
    \end{pmatrix}.
\end{align}

It is readily checked from (\ref{eq:Jacobian1}), that the two-dimensional subspace $$V := \left\{\begin{pmatrix}
    au_{n_*}\\
    b
\end{pmatrix},~a,b\in\mathbb{R}\right\},$$ is $D\beta(x_*,n_*)$-invariant. Moreover, since $\pi_1(V)= \{au_{n_*},~a\in\mathbb{R}\}$, we find that $\dim(\pi_1(V))=1$.

Due to the assumption of a complex pair of eigenvalues, there is only one two-dimensional invariant subspace, thus $V$ is precisely the subspace associated with the complex eigenvalues.

Finally, the derivative (\ref{eq:Jacobian1}) restricted to $V$ is given by
\begin{align*}
    D\beta(x_*,n_*)|_V &= \begin{pmatrix}
        u_{n_*}^T&0\\
        0&1
    \end{pmatrix}D\beta(x_*,n_*)\begin{pmatrix}
        u_{n_*}&0\\
        0&1
    \end{pmatrix}\\
    &= \begin{pmatrix}
        u_{n_*}^Tf'(x_*)u_{n_*}&0\\
        0&0
    \end{pmatrix} + \lambda_1^{-1}\begin{pmatrix}
        \varepsilon\\
        1
    \end{pmatrix}\begin{pmatrix}
        w^Tu_{n_*}&\lambda^{-1}
    \end{pmatrix}\\
    &=\begin{pmatrix}
        \lambda + \lambda_1^{-1}\varepsilon w^Tu_{n_*}&\varepsilon \lambda_1^{-1}\lambda^{-1}\\
        \lambda_1^{-1}w^Tu_{n_*}&\lambda_1^{-1}\lambda^{-1}
    \end{pmatrix},
\end{align*}
from which it follows that
\begin{displaymath}
    \det(DB(x_*,n_*)|_V) = \lambda_1 + \varepsilon\lambda^{-1}\lambda_1 w^Tu_{n_*} - \varepsilon\lambda^{-1}\lambda_1w^Tu_{n_*} = \lambda_1.
\end{displaymath}

\end{proof}

We note the properties highlighted in Proposition~\ref{prop:eigenrelation} also hold for periodic points, see \cite[Chapter~5]{tey2022minimal}.